\documentclass{amsart}
\usepackage{hyperref}
\usepackage{amssymb,xypic}
\usepackage[]{latexsym}

\usepackage[]{amsmath}
\usepackage{amsfonts} 

\usepackage{amsthm}
\usepackage[all]{xy}
\usepackage[]{mathrsfs}
\usepackage{enumerate}

\usepackage[capitalise]{cleveref}

\usepackage{color}
\usepackage[french,english]{babel}

\renewcommand{\geq}{\geqslant}

\renewcommand{\leq}{\leqslant}

\newcommand{\qf}[1]{{\langle{#1}\rangle}} % formes quadratiques
\newcommand{\C}{\mathcal{C}}
\newcommand{\co}{\mathcal {O}}
\newcommand{\s}{\sigma}
\newcommand{\mz}{\mathbb{Z}}

\newcommand{\HH}{\mathbb{H}}
\newcommand{\pfr}[1]{\mbox{$\langle\!\langle #1 ]]$}}
\newcommand{\qform}[1]{{\langle{#1}\rangle}} % formes quadratiques
\DeclareMathOperator{\End}{End}
\DeclareMathOperator{\ad}{ad}

\DeclareMathOperator{\Int}{Int}

\DeclareMathOperator{\Ad}{Ad}

\DeclareMathOperator{\id}{Id}

\DeclareMathOperator{\Nrd}{Nrd}

\DeclareMathOperator{\Trd}{Trd}
\DeclareMathOperator{\Srd}{Srd}

\DeclareMathOperator{\Sym}{Sym}
\DeclareMathOperator{\Symd}{Symd}
\DeclareMathOperator{\Skew}{Skew}
\DeclareMathOperator{\Alt}{Alt}

\newcommand{\can}{\overline{\rule{2.5mm}{0mm}\rule{0mm}{4pt}}}

\DeclareMathOperator{\Spin}{Spin}

\DeclareMathOperator{\PGO}{PGO}

\newcommand{\pff}[1]{\mbox{$\langle\!\langle #1
\rangle\!\rangle $}}

\newtheorem{lem}{Lemma}[section]
\newtheorem{prop}[lem]{Proposition}
\newtheorem{thm}[lem]{Theorem}
\newtheorem{cor}[lem]{Corollary}

\theoremstyle{remark}
\newtheorem{remark}[lem]{Remark}
\newtheorem{defi}[lem]{Definition}
\newtheorem{ex}[lem]{Example}

\title[]{ The canonical quadratic pair on a Clifford algebra and triality}  
\author{Andrew Dolphin}

\address{Universiteit Antwerpen, Departement Wiskunde-Informatica, Middelheimlaan 1, 2020 Antwerpen, Belgium}
\email{Andrew.Dolphin@uantwerpen.be}

\author{Anne Qu\'eguiner-Mathieu}
\address{Universit\'e Paris 13\\
Sorbonne Paris Cit\'e\\
LAGA - CNRS (UMR 7539)\\
F-93430 Villetaneuse, France}
\email{queguin@math.univ-paris13.fr}
\thanks{}

\subjclass[2010]{}

\date{\today} 

\begin{document}

\begin{abstract}
We define a canonical quadratic pair on the Clifford algebra of an algebra with quadratic pair over a field. 
This allows us to extend to the characteristic $2$ case the notion of trialitarian triples, from which we derive a characterization of totally decomposable quadratic pairs in degree $8$. We also describe trialitarian triples involving algebras of small Schur index. 
\end{abstract}

\maketitle

\section{Introduction}

Triality is a phenomenon that arises due to the high level of symmetry in the Dynkin diagram $D_4$.
 This symmetry is reflected in objects associated to  groups of type $D_4$, such as  $8$-dimensional quadratic forms, and degree $8$ central simple algebras with orthogonal involution. 
 More precisely, consider a degree $8$ central simple algebra $A$ over a field $F$ of characteristic different from $2$. Assume $A$ is endowed with an orthogonal involution $\sigma_A$ with trivial discriminant. The Clifford algebra $\C(A,\s_A)$, with its canonical involution $\underline{\s_A}$,  is a direct product of two central simple algebras with involution, which also have degree $8$ and are of orthogonal type, so that we actually get a triple 
  \[\bigl((A,\sigma_A),(B,\sigma_B),(C,\sigma_C)\bigr),\]
  called a trialitarian triple, and an isomorphism 
\begin{equation}\label{equation:tri}
 (\C(A,\s_A), \underline{\s_A})\simeq (B,\s_B)\times (C,\s_C)\,.\tag{$\star$}\end{equation}
 By~\cite[\S 42]{KMRT}, triality then permutes the algebras with involution in this expression. That is, \eqref{equation:tri} induces isomorphisms
 \begin{eqnarray*}(\C(B,\s_B),\underline{\s_B})&\simeq& (C,\s_C)\times (A,\s_A)\\
 (\C(C,\s_C),\underline{\s_C})&\simeq& (A,\s_A)\times (B,\s_B)\,.
\end{eqnarray*} 
In particular, it follows that the Clifford algebra, viewed as an algebra with involution, is a complete invariant for orthogonal involutions with trivial discriminant on a degree $8$ algebra.  

This trialitarian relation has proven to be extremely fruitful; roughly speaking, triality plays the same role in degree $8$ as the so-called exceptional isomorphisms in smaller degree. For instance, it can be used to characterize totally decomposable orthogonal involutions on algebras of degree $8$ (see \cite[\S42.B]{KMRT} and connected problems in \cite{BPQM}). It is related to the  classification of groups of type $D_4$ (see \cite[\S44]{KMRT} and~\cite{Gar:triality}). It makes the degree $8$ case a crucial test case for some general questions on algebras with involution, see for instance~\cite[Thm.~5.2]{QT:conic} and~\cite[\S 4]{QT:Arason}. Finally, this better understanding of the degree $8$ case can in turn be used to answer questions in larger degree; for instance, it leads to an example of a degree $16$ non-totally decomposable algebra with involution that is totally decomposable after generic splitting of the underlying algebra, see~\cite[Thm.~1.3]{QT:Pfister}. 

For fields of characteristic 2, triality is not as well studied  due to complications arising when studying quadratic forms and orthogonal groups over these fields. 
In particular, the notions of symmetric bilinear forms and quadratic forms are no longer equivalent. Over such a base, the automorphism group of a bilinear form is not semisimple anymore, so the corresponding twisted objects, in particular orthogonal involutions, cannot be used to describe such algebraic groups. 
Twisted groups of type $D$ in characteristic $2$ were initially studied by Tits, who used so-called generalised quadratic forms (see \cite{Tits:genquadforms} or \S\ref{triples.sec}), which appear to be a good replacement for hermitian forms in this setting. 
Involution-like corresponding objects, namely quadratic pairs, were introduced later in~\cite[\S5]{KMRT}. They are related to generalised quadratic forms by an adjunction process, and behave better than generalised quadratic forms, for instance under scalar extention. They provide an appropriate tool to describe groups of type $D$ over a field of characteristic different from $2$. This theory is  developed in the Book of Involutions~\cite{KMRT}, where most of the material, about involutions and quadratic pairs, their invariants, and relations to algebraic groups, is developped over a field of arbitrary characteristic. 
However, Chapter X, about Trialitarian Central Simple Algebras, is one of the rare exceptions  in~\cite{KMRT}; the base field is assumed to be of characteristic not $2$ in that section. 

That the group $\mathrm{Spin}_8$ has this exceptionally large group of outer automorphisms is true independent of the characteristic of the underlying field, and hence some trialitarian relation should hold for quadratic pairs in characteristic $2$ also.
 One has a notion of a Clifford algebra of a quadratic pair, and we again  have that the Clifford algebra of a quadratic pair with trivial discriminant is the direct product of two degree $8$ central simple algebras with involution (see \cite[\S7 and \S8]{KMRT}). However, in order to fully  recapture the trialitarian relation, one also needs that the Clifford algebra be equipped with a canonical quadratic pair, not just a canonical involution.
A definition of this canonical quadratic pair is briefly sketched out in  \cite[p.149]{KMRT}, in the particular case where $A$ is split, of degree divisible by $8$, and endowed with a hyperbolic quadratic pair. From this, one can define a canonical quadratic pair in the more general case  via Galois descent. However this definition is not easy to use, and the lack of a  `rational'  definition, that is a definition that avoids the use of Galois descent, is one reason why  the results in \cite[Chapt.~X]{KMRT} are restricted to fields of characteristic different from $2$ (see  \cite[Chapt.~X, Notes]{KMRT} for more details). 

The main purpose of this paper is to provide a rational definition for the canonical quadratic pair of the Clifford algebra of an algebra with quadratic pair, see~\cref{def}. We use as a crucial tool the Lie algebra structures described in~\cite[\S 8.C]{KMRT}. We also provide an explicit description of this canonical quadratic pair in the split case in \cref{split.section}. With this in hand, we may extend to arbitrary fields the main results of~\cite[(\S 42)]{KMRT}. In particular, we define a notion of trialitarian triple, and describe the trialitarian action in~\cref{sec:tri}, and we characterize totally decomposable algebras with quadratic pair in degree $8$ in~\cref{thm:totdecomp}. The last section describes all trialitarian triples of small enough Schur index, see~\cref{triples.sec}. Partial results in this direction were previsouly obtained by Knus and Villa~\cite[\S\,7]{trialityquat}. 

We first recall some notation and basic results (\S 2.1 to 2.3), and make some preliminary observations on quadratic pairs and tensor products (\S~\ref{prod.section}). 

We are grateful to Seidon Alsaody, Philippe Gille and Jean-Pierre Tignol for insightful discussions. 

\section{Preliminaries}

Throughout the paper, $F$ is a field. 
We refer the reader to~\cite{pierce:1982} as a general reference on central simple algebras, \cite{KMRT} for involutions and quadratic pairs and~\cite{Elman:2008} for hermitian, bilinear and quadratic forms.

\subsection{Algebras with involution} 
\label{AI.sec}
Let $A$ be a central simple algebra of degree $n$ over $F$. 
To all $a\in A$, we associate its reduced characteristic polynomial 
$$ \mathrm{Prd}_{A,a}(X)= X^n- s_1(a)X^{n-1}+s_2(a)X^{n-2}-\ldots+(-1)^ns_n(a),$$ 
which is the characteristic polynomial of $a\otimes 1\in A\otimes_F \Omega\simeq M_n(\Omega)$, where $\Omega$ is an algebraic closure of $F$, see~\cite[\S 16.1]{pierce:1982}. The coefficients of  $\mathrm{Prd}_{A,a}$ belong to $F$; 
 $s_1(a)$ and  $s_n(a)$ are the {reduced trace} and the {reduced norm of $a$}, respectively denoted by $\Trd_A(a)$ and $\Nrd_A(a)$, and $s_2(a)$ is denoted by $\Srd_A(a)$.

All the involutions considered in this paper are $F$-linear. 
If the algebra $A$ is split, that is, $A\simeq \End_F(V)$, an $F$-linear involution on $A$ is the adjoint of a nondegenerate symmetric or skew-symmetric bilinear form $b:V\times V\rightarrow F$, uniquely defined up to a scalar factor. We denote this algebra with involution by $\Ad_b$. The involution is symplectic if $b$ is alternating, and orthogonal if $b$ is symmetric and non-alternating. 

Let $\sigma$ be an $F$-linear involution on $A$. 
We use the same notation as in~\cite[\S 2.A]{KMRT} for the subvector spaces $\Sym(A,\sigma)$, $\Symd(A,\sigma)$, $\Skew(A,\sigma)$ and $\Alt(A,\sigma)$ of symmetric, symmetrized, skew-symmetric and alternate elements, respectively.
Recall $\Sym(A,\sigma)=\Symd(A,\sigma)$ if the base field has characteristic different from $2$, while in characteristic $2$, $\Symd(A,\sigma)=\Alt(A,\sigma)$ is a strict subspace of $\Sym(A,\sigma)=\Skew(A,\sigma)$, and they have dimension $\frac{n(n-1)}{2}$ and $\frac{n(n+1)}{2}$, respectively.  
Still assuming the base field has characteristic $2$, one may prove that the involution $\sigma$ is symplectic if and only if $1$ is a symmetrised element, or equivalently all symmetric elements have reduced trace $0$~\cite[(2.6)]{KMRT}. In particular, in characteristic $2$, a tensor product of involutions with at least one symplectic factor always is symplectic.  
In characteristic different from $2$, a tensor product of involutions is symplectic if and only if there are an odd number of symplectic involutions in the product (see \cite[(2.23)]{KMRT}).

An $F$-quaternion algebra has a basis $(1,u,v,w)$ such that
$$u(1-u) =a, v^2=b\textrm{ and }w=uv=v(1-u)\,$$
for some  $a\in F$ with $4a\neq -1$  and $b\in F^\times$
 (see  \cite[Chap.~IX, Thm.~26]{Albert:1968}); any such basis is called a quaternion basis throughout this paper. 
 Conversely, for  $a\in F$ and $b\in F^\times$
 the above relations uniquely determine an $F$-quaternion algebra, which we denote by $H=[a,b)$. 
 If the characteristic of $F$ is different from $2$, substituting $i=u-\frac{1}{2}$ and $j=v$ give the more usual basis of $Q$, $\{1,i,j,ij\}$, with $i^2=c$, $j^2=d$ for $c,d\in F^\times$ and $ij=-ji$. In this case we denote $Q$ by $(c,d)$.
  
 Recall $H=[a,b)$ has a unique symplectic involution, called the canonical involution, which is determined by the conditions that $\overline{u}=1-u$ and $\overline{v}=-v$. 
Considering $H$ as a $4$-dimensional vector space over $F$, we may view $\Nrd_H$ as a $4$-dimensional quadratic form over $F$, which we call the norm form of $H$.

\subsection{Quadratic forms and their Clifford algebras} 

For $b\in F^\times$, we denote the $2$-dimensional  symmetric bilinear form $(x_1,x_2)\times (y_1,y_2)\mapsto x_1y_1+ bx_2y_2$
by $\qf{1,b}^{bi}$. 
For a nonnegative integer m, by an m-fold bilinear Pfister form, we mean a nondegenerate symmetric bilinear form isometric to a tensor product of $m$ binary symmetric bilinear forms representing $1$; we use the notation 
$\pff{b_1,\dots, b_m}\simeq\qf{1,-b_1}^{bi}\otimes\dots\otimes\qf{1,-b_m}^{bi}$. 

Let $q:\,V\rightarrow F$ be a quadratic form and denote  its polar form by $b_q$, defined by $b_q(x,y)=q(x+y)-q(x)-q(y)$, so that $b_q(x,x)=2q(x)$ for all $x\in V$. It is a symmetric  bilinear form over $V$, and further, it is alternating if $F$ is of characteristic $2$. The quadratic form $q$ is called nonsingular if its polar form is nondegenerate. In characterictic $2$, this implies it is hyperbolic~\cite[Prop. 1.8]{Elman:2008}.
For all $b_1,b_2\in F$, we let $[b_1,b_2]$ be the  quadratic form $(x,y)\rightarrow b_1x^2+xy+b_2y^2$.
This form is nonsingular if and only if $-1\neq 4 b_1b_2$.
Note that the hyperbolic plane $\HH=[0,0]$ satisfies $\HH\simeq [a,0]\simeq [0,a]$ for all $a\in F$, so for every nonsingular $2$-dimensional quadratic form we may chose a presentation $[a,b]$ with $a\not=0$. Any nonsingular quadratic form $\phi$ over $F$ has a Witt decomposition $\phi \cong \phi_{an} \perp i_W(\phi)\times \HH$,  where $\phi_{an}$ is the anisotropic part and the integer $i_W(\phi)$ is the Witt index of $\phi$.  

To a quadratic form $q:\,V\rightarrow F$ and a symmetric bilinear form $b:\,W\times W\rightarrow F$, one associates the quadratic form, denoted by $b\otimes q$ and defined on $W\otimes V$ by $$(b\otimes q)(w\otimes v)=b(w,w)q(v) \,\textrm{ for all $w\in W,v\in V$}\,. $$
(see~\cite[p.51]{Elman:2008}).
For any positive  integer $m$, by an $m$-fold quadratic Pfister form we mean a quadratic form that is isometric to the tensor product of an $(m-1)$-fold bilinear Pfister form and a nonsingular binary quadratic form representing $1$. We use the notation 
$\pfr{b_1,\dots,b_{m-1},c}=\pff{b_1,\ldots, b_{m-1}}\otimes [1,c]$.
Our definition is equivalent to the definition in~\cite[\S 9.B]{Elman:2008}, even though we use a different notation in characteristic different from $2$. 
 Pfister forms are either anisotropic or hyperbolic (see \cite[(9.10)]{Elman:2008}). 

Recall the Clifford algebra of a quadratic space $(V,q)$ is a factor of the tensor algebra $T(V)$ by the ideal  $I(q)$ generated by elements of the form $v\otimes v - q(v)\cdot 1$ for $v\in V$. 
It has a natural $\mathbb{Z}/2\mathbb{Z}$-gradation, and the subalgebra $\C_0(q)$ of degree $0$ elements is called the even Clifford algebra.
The identity map on $V$ extends to an involution on $\C(q)$ and $\C_0(q)$ called the canonical involution and denoted $\underline{\sigma_q}$.
If $q$ is nonsingular, the center of $\C_0(q)$ is a quadratic \'etale extension of $F$. It is determined by a class which belongs to the multiplicative group of square classes $F^\times/F^{\times 2}$ in characteristic different from $2$ and the additive group $F/\wp(F)$ in characteristic $2$, where $\wp(F)=\{a^2+a\mid a\in F\}$. In both cases, we will refer to this class as the discriminant of $q$ and denote it by $\Delta(q)$. If $F$ has characteristic $2$ and $q\simeq [a_1,{b_1}]\bot\ldots\bot [{a_n},b_n]$ for some $a_i,b_i$ in $F$, then $\Delta(q)$ is the class of $a_1b_1+\dots+a_nb_n$ in $F/\wp(F)$. 

\begin{ex} \label{ex:cliffex} 
For ease of reference in the sequel, we give an explicit description of the even Clifford algebra of a nonsingular quadratic form over a field of characteristic $2$. Given such a form $q$, with polar form $b$, pick a decomposition \[q \simeq [a_1,b_1]\perp\ldots\perp [a_m,b_m],\] and let $(e_i,e_i')_{1\leqslant i \leqslant m}$ be the corresponding symplectic basis of the underlying vector space $V$. That is,  for all $i$ with $1\leqslant i \leqslant m$, we have  $q(e_i)=a_i$, $q(e_i')=b_i$, $b(e_i,e'_i)=1$ and $b(e_i,e_j)=b(e'_i,e'_j)=b(e_i,e'_j)=0$ for all $i\neq j$. We may assume $a_i\not=0$ for all $i$. 

The elements $u_i= e_ie'_i$ and  $v_i=e_ie_m$, for $i\in\{1,\dots, m-1\},$
belong to the even part $\C_0(q)$ of the Clifford algebra, and satisfy 
\[u_i(1+u_i)= a_ib_i,\ v_i^2  =  a_ia_m,\ u_iv_i=v_i(1+u_i).\]
They generate pairwise commuting quaternion subalgebras. Further, we have that 
$$\underline{\sigma_q}(u_i) = 1+u_i \quad \textrm{ and }\quad \underline{\sigma_q}(v_i)=v_i\,.$$

We have that
 \[\xi= \sum_{i=1}^m e_ie'_i\] also belongs to $\C_0(q)$, commutes with $u_i$ and $v_i$ for all $i$, $1\leq i\leq m-1$, and satisfies 
$ \xi^2 = \xi +\Delta(\rho)\,.$
Hence, $F[\xi]$ is a quadratic \'etale extension of $F$, central in $\C_0(q)$. 
We also  have that 
$$\underline{\sigma_q}(\xi)=\left\{\begin{array}{l}1+ \xi \text{ if }m\text{ is odd,}\\\xi\text{ if }m\text{ is even.}\\\end{array}\right.$$
So we finally get
\[\bigl(\C_0(q),\underline{\sigma_q}\bigr)\simeq \otimes_{i=1}^{m-1}(Q_i,\can)\otimes (F[\xi],\gamma)\,,\] 
    where $Q_i=[a_ib_i,a_ia_m)$, $\can$ stands for the canonical involution, $F[\xi]$ is the centre of $\C_0(q)$, and $\gamma$ is the identity if $m$ is even and the non-trivial $F$-automorphism of $F[\xi]$ if $m$ is odd. 
 \end{ex}

We finish this subsection with a characteristic free version of \cite[(Example 9.12)]{Kneb} which we require in the sequel. The proof is similar, but we provide the full details for convenience.
\begin{prop}\label{prop:8dim}
Let $q$ be an $8$-dimensional nonsingular quadratic form over $F$ with trivial discriminant  and $\mathrm{ind}(\C_0(q))\leqslant 2$. Then there exists a $4$-dimensional symmetric bilinear form $B$ and a $2$-dimensional nonsingular form $\phi$ over $F$ such that $q\simeq B\otimes \phi$. 
\end{prop}
\begin{proof}
Over its function field, $q$ is Witt equivalent to an Albert form, which is isotropic by~\cite[(16.5)]{KMRT}. Therefore, the first Witt index of $q$ is at least $2$. Choose a quadratic separable extension $F(u)/F$ where $u-u^2=a$ for some $a\in F$ with $-1\neq 4a$ such that $q$ becomes isotropic after extending scalars to $F(u)$. By~ \cite[(25.1)]{Elman:2008}, the Witt index of $q$ over $F(u)$ is at least $2$. 
Hence by \cite[Chapter V, (4.2)]{Baeza:1982}, we have that $q\simeq c\pfr{b,a}\perp q'$ for some $b,c\in F^\times$ and a $4$-dimensional  nonsingular quadratic form $q'$ over $F$. Since $q$ has trivial discriminant, it follows that $q'$ also have trivial discriminant and hence $q'$ is similar to a Pfister form, which we denote by $\pi$. The form $\pfr{b,a}\perp -\pi$ is Witt equivalent to an Albert form, and has the same Clifford invariant as $q$, of index $\leq 2$. Therefore, its Witt index is at least $2$, and by~\cite[(24.2)]{Elman:2008}, there exist symmetric bilinear forms $B'$ and $B''$ and $d\in F$ such that  $\pfr{b,a}\simeq B'\otimes [1,d]$ and $\pi\simeq B''\otimes [1,d]$. In particular, we have $q\simeq B\otimes [1,d]$ for some symmetric bilinear form $B$ over $F$. 

\end{proof}

\subsection{Quadratic pairs and their Clifford algebras}
In arbitrary characteristic, algebraic groups of type $D$ can be described in terms of quadratic pairs. 
For the reader's convenience, we recall here some basic facts on quadratic pairs which can be found in~\cite[\S 5,7.B, 8.B]{KMRT}, and which are used throughout the paper. 

A quadratic pair on a central simple algebra $A$ is a couple $(\sigma,f)$, where $\sigma$ is an $F$-linear involution on $A$, with $\Sym(A,\sigma)$ of dimension $\frac{n(n+1)}{2}$  and $f$ is a so-called semi-trace on $(A,\sigma)$. That is, $f$ is an $F$-linear map $f:\,\Sym(A,\sigma)\rightarrow F$ such that \[f(x+\sigma(x))=\Trd_A(x)\text{ for all }x\in A.\] In characteristic different from $2$, the dimension condition guarantees that the involution is of orthogonal type, and one may check that there is a unique semi-trace on $(A,\sigma)$ given  by $f(x)=\frac 12 \Trd_A(x)$ for all $x\in \Sym(A,\sigma)$. Therefore quadratic pairs and orthogonal involutions are equivalent notions when the  characteristic is not $2$. 
Conversely, in characteristic $2$, the existence of a semi-trace implies $\sigma$ is symplectic. Indeed, since $\Trd_A(c)=f(c+\sigma(c))=f(2c)=0$ for all $c\in \Sym(A,\sigma)$, the reduced trace vanishes on $\Sym(A,\sigma)$. 

One may easily check that for all $\ell\in A$ such that $\ell+\sigma(\ell)=1$, the $F$-linear map defined  by $f_\ell(s)=\Trd_A(\ell s)$ for all $s\in\Sym(A,\sigma)$ is a semi-trace on $(A,\sigma)$. Conversely, it is proved in~\cite[(5.7)]{KMRT} that any semi-trace $f:\Sym(A,\sigma)\rightarrow F$ coincides with $f_\ell$ for some  $\ell\in A$ satisfying $\ell+\sigma(\ell)=1$. We say that the element $\ell$ gives or determines the semi-trace $f_\ell$. Two distinct such elements $\ell$ and $\ell'$ determine the same semi-trace if and only if they differ by an alternating  element, that is $\ell-\ell'=x-\sigma(x)$ for some $x$ in $A$.

Let $(V,q)$ be a nonsingular quadratic space over the field $F$. The polar form $b_q$ of $q$ induces an involution $\sigma_q=\ad_{b_q}$ on $A=\End_F(V)$, and one may prove that $(\End_F(V),\sigma_q)\simeq (V\otimes V,\varepsilon)$, where $\varepsilon$ is the exchange involution, defined by $\varepsilon(x\otimes y)=y\otimes x$. Moreover, there exists a unique semi-trace $f$ defined on $\Sym(V\otimes V, \varepsilon)$ and satisfying $f(x\otimes x)=q(x)$ for all $x\in V$. 
Under the isomorphism above, $f$ defines a semi-trace $f_q$ on $(\End_F(V),\sigma_q)$. The quadratic pair $\ad_q=(\sigma_q,f_q)$ is called the adjoint of $q$, and we use the notation $\Ad_{q}$ for the algebra with quadratic pair $(\End_F(V),\sigma_q,f_q)$. As explained in~\cite[(5.11)]{KMRT}, any quadratic pair on a split algebra $\End_F(V)$ is the adjoint of a nonsingular quadratic form $q$ on $V$. 

Let $(A,\sigma,f)$ be an $F$-algebra with quadratic pair.  For further use, we briefly recall the definition of the discriminant and the Clifford algebra of 
$(A,\sigma,f)$, as featured in~\cite{KMRT}. Assume $A$ has even degree $n=2m$, and $\ell\in A$ determines the semi-trace $f$. 
If the characteristic of $F$ is different from $2$, the discriminant of $(\sigma,f)$ is given by  $$\textrm{disc}(\sigma, f)= (-1)^m\Nrd_A(a) \in F^\times /F^{\times 2}$$ for any element $a\in \Alt(A,\sigma)\cap A^\times$.
If $F$ is of characteristic $2$, the discriminant of $(\sigma,f)$ is given by 
$$ \textrm{disc}(\sigma, f)= \textrm{Srd}(\ell) +\frac{m(m-1)}{2} \in F/\wp(F)\,. $$
It extends the discriminant  invariant for quadratic forms. That is if $q$ is a nonsingular quadratic form, then
we have $\textrm{disc}(\ad_q)=\Delta(q)$. See \cite[\S7.B]{KMRT} for more details. 

Let $\mathrm{Sand}:\,\underline A\otimes\underline A\rightarrow \End_F(A)$ denote the sandwich linear map, as defined in~\cite[(3.4)]{KMRT}), where $\underline A$ denotes $A$ viewed as an $F$-vector-space. 
The Clifford algebra  $\C(A,\sigma,f)$ of the algebra with quadratic pair $(A,\sigma,f)$ is the factor of the tensor algebra $T(\underline{A})$:
$$\C(A,\sigma,f)= \frac{T(\underline{A})}{J_1({\sigma,f})+ J_2({\sigma,f})}$$
where 
\begin{itemize}
\item[$(1)$]
$J_1({\sigma,f})$ is the ideal generated by all the elements of the form $s-f(s)\cdot 1$ for $s\in \underline{A}$ such that $\sigma(s)=s$;
\item[$(2)$]
 $J_2(\sigma,f)$ is the ideal generated by all elements of the form $u-\mathrm{Sand}(u)(\ell)$   
for 
$u\in \underline{A}\otimes \underline{A}$  such that $\sigma_2(u)=u$ and 
where $\sigma_2$ is defined by the condition 
$$  \mathrm{Sand}(\sigma_2(u))(x) =\mathrm{Sand}(u)(\sigma(x)) \quad \textrm{for } u\in \underline{A}\otimes \underline{A},\  x\in\underline{A} \,.$$ 
\end{itemize}

The centre of the Clifford algebra of $(A,\sigma,f)$ is a quadratic \'etale extension, which is related to the discriminant as follows. In characteristic different from $2$, the centre of the Clifford algebra is given by $F\bigl({\sqrt{\mathrm{disc}(\sigma,f)}}\,\bigr)$ (see ~\cite[(8.25)]{KMRT}). In characteristic $2$, the centre is given by $F(u)$ where $u^2+u=\mathrm{disc}(\sigma,f)$ (see ~\cite[(8.27), (8.28)]{KMRT}). In either case, if the discriminant is trivial, this  quadratic \'etale extension is isomorphic to $F\times F$ and the Clifford algebra splits into two components. 

In addition, $\C(A,\sigma,f)$ is endowed with a canonical involution, denoted by $\underline\sigma$, and induced by the involution of $T(\underline A)$ acting as $\sigma$ on $\underline A$, that is $$ \underline{\sigma} (a_1\otimes \ldots \otimes a_r) = \sigma(a_r)\otimes \ldots \otimes \sigma(a_1)\quad \textrm{for all } a_1,\ldots, a_r\in\underline A\,.$$ 
The Clifford algebra of $(A,\sigma,f)$ extends the even Clifford algebra for quadratic spaces, that is if $(A,\sigma,f)\simeq\Ad_q$ for some nonsingular quadratic form $q$, there is a canonical isomorphism between $\C(A,\sigma,f)$ and $\C_0(q)$, and the canonical involution $\underline \sigma$ corresponds to $\underline{ \sigma_q}$ under this isomorphism. 
See \cite[\S8]{KMRT} for more details.

\subsection{Semi-traces and tensor products}\label{prod.section}
Consider an embedding \[i:\,(A,\sigma)\rightarrow(D,\rho),\] where $(A,\sigma)$ and $(D,\rho)$ are two algebras with symplectic involution. Any element $\ell\in A$ such that $\ell+\sigma(\ell)=1$ maps to an element $i(\ell)\in D$ such that $i(\ell)+\rho(i(\ell))=1$. In addition, symmetrised elements in $(A,\sigma)$ map to symmetrised elements in $(D,\rho)$. Therefore, to any semi-trace $f=f_\ell$ on $(A,\sigma)$, we may associate a well defined semi-trace $g=f_{i(\ell)}$ on $(D,\rho)$. Clearly, the semi-trace $g$ depends not only on $f$, but also on the embedding $i$. When $i$ is canonical, we   forget the embedding and use the same notation $f_\ell$ for both semi-traces. 
This correspondence is not  extending or restricting the semi-trace viewed as a map, even though $i$ maps symmetric elements in $(A,\sigma)$ to symmetric elements in $(D,\rho)$. 
For instance, if $D$ is $F$-central and $A$ has centre $Z_A$, then $f$ is $Z_A$-linear with values in $Z_A$, while $g$ is $F$-linear with values in $F$, and $Z_A$ may be strictly larger than $F$. We will refer to $g$ as the semi-trace induced by $f$ on $(D,\rho)$. 

Let $(B,\tau)$ be an algebra with involution, assumed to be orthogonal if $F$ is of characteristic different from $2$ and $(A,\sigma,f)$ an algebra with quadratic pair. 
The involution $\tau\otimes \sigma$  is then orthogonal if the characteristic of $F$ is different from $2$ and  symplectic otherwise. 
Therefore, the construction above applies to the canonical embedding $(A,\sigma)\subset (B\otimes A,\tau\otimes\sigma)$, so that $f$ induces a semi-trace $f_\star$ on $(B\otimes A, \tau\otimes \sigma)$. For all $b\in \Sym(B,\tau)$ and $a\in \Sym(A,\sigma)$, $b\otimes a\in\Sym(B\otimes A, \tau\otimes \sigma)$ and we have 
\[f_\star(b\otimes a)=\Trd_{B\otimes A}\bigl((1\otimes \ell) (b\otimes a)\bigr)=\Trd_B(b)\Trd_A(\ell a)=\Trd_B(b) f(a),\] 
where $\ell$ is an element defining the semi-trace $f$. 
In~\cite[(5.18)]{KMRT}, it is proved that this condition characterizes the semi-trace  $f_\star$. 
This construction defines a tensor product \[(B,\tau)\otimes(A,\sigma, f)=(B\otimes A,\tau\otimes \sigma,f_\star).\] 

One may check that this tensor product corresponds to the usual one in the split case, that is $\Ad_b\otimes \Ad_\rho=\Ad_{b\otimes\rho}$ for all nondegenerate symmetric bilinear forms $b$ and nonsingular quadratic forms $\rho$, see \cite[(5.19)]{KMRT}. 
In addition, it is associative, that is 
$$ \bigl((C,\gamma)\otimes (B,\tau)\bigr)\otimes(A,\sigma,f) \simeq  (C,\gamma)\otimes \bigl((B,\tau)\otimes (A,\sigma,f)\bigr),$$
for all algebra with involution $(C,\gamma)$, see~\cite[(5.3)]{dolphin:totdecomp}.
In particular 
we may write $(C,\gamma)\otimes(B,\tau)\otimes(A,\sigma,f)$ without any ambiguity. We say that $(A,\sigma,f)$ is totally decomposable if there exist $F$-quaternion algebras with involution $(Q_i,\sigma_i)_{1\leqslant i\leqslant n-1}$ and an $F$-quaternion algebra with quadratic pair $(Q_n,\sigma_n,g)$ such that \[(A,\sigma,f)\simeq \left(\bigotimes_{i=1}^{n-1}(Q_i,\sigma_i)\right)\otimes (Q_n,\sigma_n,g).\]

Consider now two algebras with symplectic involutions $(A,\sigma)$ and $(B,\tau)$. If the characteristic of $F$ is different from $2$, the involution $\tau\otimes \sigma$ is orthogonal, and hence there is a unique associated semi-trace. If the characteristic of $F$ is $2$, then we have $\Trd_B(b)=0$ for all $b\in\Sym(B,\tau)$ and the formula above shows that, given an arbitrary semi-trace $f$ on $\Sym(A,\sigma)$, the induced semi-trace $f_\star$ on $\Sym(B\otimes A,\tau\otimes\sigma)$ vanishes on \[\Sym(B,\tau)\otimes \Sym(A,\sigma)\subset \Sym(B\otimes A,\tau\otimes\sigma).\]
Again, this condition characterizes $f_\star$, see~\cite[(5.20)]{KMRT}, and in particular, $f_\star$ does not depend on the choice of $f$ when $\tau$ also is symplectic. We now extend this result to a product with $r$ factors. 
 \begin{prop}\label{canprod}
 Assume $F$ has characteristic $2$ and let
$(A_i,\sigma_i)_{1\leq i\leq r}$ be $r$ algebras with symplectic involutions for some $r\geq 2$. 
There exists a unique semi-trace $$f_\otimes:\Sym\left(\bigotimes_{i=1}^r(A_i,\sigma_i)\right)\rightarrow F$$ 
such that $$f_\otimes|_{\bigotimes_{i=1}^r \Sym(A_i,\sigma_i)} =0\,.$$ 
\end{prop}
We call $f_\otimes$  the {canonical semi-trace on the tensor product of $F$-algebras with symplectic involution $\bigotimes_{i=1}^r(A_i,\sigma_i)$}. In the sequel, we will also denote by $f_\otimes$ the unique semi-trace on the algebra with orthogonal involution $\bigotimes_{i=1}^r(A_i,\sigma_i)$ when $F$ has characteristic different from $2$ and $r$ is even. 

\begin{proof} Recall $F$ has characteristic $2$. For any algebra with involution $(A,\sigma)$, given $a\in \Symd(A,\sigma)$ and  $s\in \Sym(A,\sigma)$ such that $s$ and $a$ commute, the product $as$ belongs to $\Symd(A,\sigma)$.  From this and \cite[(5.17)]{KMRT}, an induction argument shows  that 
$$  \Sym\left(\bigotimes_{i=1}^r(A_i,\sigma_i)\right) = \Symd \left(\bigotimes_{i=1}^r(A_i,\sigma_i)\right) + \bigotimes_{i=1}^r \Sym(A_i,\sigma_i)\,.  $$
The uniqueness of the semi-trace $f_\otimes$ follows as the value of any semi-trace on $\Symd$ is fixed.

It remains to prove the existence of such a semi trace. Let $(B,\tau)=\bigotimes_{i=1}^{r-1} (A_i,\sigma_i)$; since $\tau$ is symplectic, we have $\Trd_B(b)=0$ for all $b\in\Sym(B,\tau)$. Pick an arbitrary semi-trace $f$ on $(A_r,\sigma_r)$ and consider the tensor product 
\[(B,\tau)\otimes(A_r,\sigma_r, f)=(B\otimes A_r,\tau\otimes \sigma_r,f_\star).\] The formula above shows that $f_\star$ vanishes on $\Sym(B,\tau)\otimes \Sym(A_r,\sigma_r)$ which contains $\bigotimes_{i=1}^r \Sym(A_i,\sigma_i)$, hence it satisfies the required condition. 
\end{proof}

\begin{remark}\label{part}
The proof actually shows that $f_\otimes$ is the semi-trace induced by an arbitrary semi-trace on one of the factors $(A_i,\sigma_i)$, that is 
\[\left( A_1\otimes\ldots\otimes A_r ,\sigma_1\otimes\ldots\otimes\sigma_r, f_\otimes \right) =\bigotimes_{1\leq k\leq r, k\not =i} (A_k,\sigma_k)\,\otimes (A_i,\sigma_{i},f_i) \,, \]
for any choice of $i$ and of a semi-trace $f_i$ on $(A_i,\sigma_i)$. 
More generally, given a non-trivial partition $I\cup J=\{1,\ldots, r\}$, that is $I\neq \emptyset\neq J$, $I\cap J=\emptyset$, we have 
\[\bigl( A_1\otimes\ldots\otimes A_r ,\sigma_1\otimes\ldots\otimes\sigma_r, f_\otimes \bigr)\simeq (A_I,\sigma_I)\otimes (A_J,\sigma_J,f_J),\]
where, for all subset $S\subset \{1,\ldots, r\}$, $(A_S,\sigma_S)=\bigotimes_{i\in S}(A_i,\sigma_i)$, and $f_J$ is an arbitrary semi-trace on $(A_J,\sigma_J)$. 
\end{remark}

\begin{remark}\label{differentdecomp}
Note that the semi trace $f_\otimes$ on the algebra with involution $(A,\sigma)\simeq \bigotimes_{i=1}^r(A_i,\sigma_i)$  does depend on the choice of $F$-algebras with involution $(A_i,\sigma_i)$ in  the decomposition.
Assume the characteristic of $F$ is $2$. 
 Since the canonical involution on a quaternion algebra is symplectic, for any $2$ quaternion $F$-algebras $Q_1$ and $Q_2$ we have
$$(Q_1,\can)\otimes (Q_1,\can)\simeq (Q_2,\can)\otimes (Q_2,\can)\simeq (M_4(F),\tau)$$
where  $\tau$ is the unique symplectic (and hyperbolic) involution on the $F$-algebra of $4\times 4$-matrices over $F$. 
 However for $i=1,2$  we have that 
$$ (Q_i\otimes Q_i,\can\otimes \can,f_\otimes)\simeq \Ad_{\Nrd_{Q_i}}$$ 
by \cite[(2.9)]{dolphin:conic}, and $\Ad_{\Nrd_{Q_1}}\simeq \Ad_{\Nrd_{Q_2}}$ holds if and only if $Q_1\simeq Q_2$. 
More generally, let $\pi_1$ and $\pi_2$ be $n$-fold Pfister forms. Then for $i=1,2$, $\Ad_{\pi_i}$  is a totally decomposable quadratic pair on the $2^n\times2^n$ matrix algebra $\mathbb{M}_{2^n}(F)$. By  \cite[(6.2)]{dolphin:totdecomp}, there exist quaternion algebras $Q_{1,i},\ldots, Q_{n,i}$ such that   
$\mathbb{M}_{2^n}(F)= \bigotimes_{j=1}^n Q_{j,i}$ and the quadratic pair $\Ad_{\pi_i}$ is the product of the canonical involutions on $Q_i$ together with  the semi-trace  $f_\otimes$. However, $\Ad_{\pi_1}\simeq \Ad_{\pi_2}$ if and only if $\pi_1\simeq \pi_2$. Therefore the semi-trace $f_\otimes$  depends on the choice of the quaternion algebras in the decomposition of $\mathbb{M}_{2^n}(F).$
\end{remark}

\section{Canonical quadratic pair on a Clifford algebra}\label{sec:canqp}

Throughout this section, $(A,\sigma,f)$ is an algebra with quadratic pair. We assume $A$ has degree $n=2m$, with $m$ even, and $m\equiv 0\mod 4$ if $F$ is of characteristic different from $2$.
Under this assumption on the degree of $A$, the canonical involution $\underline \sigma$ of the Clifford algebra $\C=\C(A,\sigma,f)$ has symplectic type in characteristic $2$, and orthogonal type in characteristic different from $2$, 
see
~\cite[(8.12)]{KMRT}. The purpose of this section is to define a semi-trace $\underline{f}$ on $\bigl(\C,\underline\sigma\bigr)$, which we call the canonical semi-trace of the Clifford algebra, provided $A$ satisfies the above conditions and 
has degree $2m\geq 8$. We also give an explicit description of $\underline f$ in the split case, with particular attention to the adjoint of a $3$-fold Pfister form. 

\subsection{Definition of the canonical semi-trace} \label{def}
Consider the $F$-linear canonical map $c:\,A\rightarrow \C$  induced by $A\rightarrow \underline{A}\rightarrow T(\underline{A})$.
By \cite[(8.16)]{KMRT} we have \[c(x)+\underline{\sigma}(c(x))=\Trd_A(x).\] 
Hence we have 
\begin{equation}
\label{trace.eq}
\text{for all }x\in A,\quad\left\{\begin{array}{l}
c(x)\in \Skew(\C,\underline{\sigma})\text{ if and only if }\Trd_A(x)=0,\\
c(x)+\underline{\sigma}(c(x))=1\text{ if and only if }\Trd_A(x)=1.\\
\end{array}\right. 
\end{equation}

The main result in this section is the following: 
\begin{prop}\label{prop:cansemi}
 Assume $A$ has degree $2m\geq 8$, with $m$ even and further that $m\equiv  0\mod 4$ if the characteristic is different from $2$. Then  $\underline \sigma$
 is symplectic in characteristic $2$ and orthogonal otherwise. For all $\lambda\in A$ with  $\Trd_A(\lambda)=1$, the element $c(\lambda)$ defines a semi-trace on $(\C,\underline{\sigma})$, which does not depend on the choice of $\lambda$ among reduced trace $1$ elements of $A$. \end{prop}

\begin{proof} By~(\ref{trace.eq}), as $\Trd_A(\lambda)=1$ we have that $c(\lambda)+\underline{\sigma}(c(\lambda))=1$, and hence the linear form which maps $s\in\Sym(\C,\underline\sigma)$ to $\Trd_\C(c(\lambda)s)$ is a semi-trace on $(\C,\underline{\sigma})$. Moreover, two elements $\lambda$ and $\lambda'$ of reduced trace $1$ differ by a trace $0$ element $\mu$, hence applying again~(\ref{trace.eq}), we have $c(\mu) = c(\lambda)- c(\lambda') \in \Skew(\C,\underline{\sigma})$. The elements $c(\lambda)$ and $c(\lambda')$ define the same semi-trace on  $(\C,\underline{\sigma})$ if and only if  $c(\mu)\in \Alt(\C,\underline{\sigma})$; this follows from the next lemma.
\begin{lem}\label{lem:imc}
Assume $A$ has degree at least $6$. In the Clifford algebra $\C$, we have \[c(A)\cap \Skew(\C,\underline{\sigma})  = c(A)\cap \Alt(\C,\underline{\sigma})\,.\]
\end{lem}

The result is trivial if the characteristic of $F$ is different from $2$, as in this case $\Skew(\C,\underline{\sigma})  = \Alt(\C,\underline{\sigma})$. Assume now that the characteristic of $F$ is  $2$.
The inclusion of the left  hand side in the right is clear. We now show the converse. It suffices to show the result in the case where $A$ is split.

Assume that $(A,\sigma,f)= \Ad_q$, for some quadratic form $q:V\rightarrow F$, and $\C=\C_0(q)$. Let $q= [a_1,b_1]\perp \ldots \perp [a_m,b_m]$ be a decomposition of $q$ and $e_1, e_1', \ldots, e_m, e_m'$ the symplectic associated basis. 
As explained in~\cite[p.95-96]{KMRT}, the set $\{1\}\cup \{e_ie_j,e_i'e_j'\mid 1\leqslant i< j \leqslant m\}\cup\{ e_ie_j' \mid  1\leqslant i,j\leqslant m \}$ is a basis of $c(A)$. 
All elements in this basis are (skew-)symmetric except for the elements $e_ie'_i$  for all  $i\in \{1,\ldots, m\}$.
However, we have $e_ie'_i +\underline{\sigma}(e_ie'_i)=1$, hence $e_ie_i' + e_{i+1}e'_{i+1}$ is symmetric for all $1\leqslant i\leqslant m-1$.
By \cite[(8.17)]{KMRT}, $c(A)_0=c(A)\cap \Skew(\C,\underline{\sigma})$ has codimension $1$ in $c(A)$. 
Therefore,
$$\{1\}\cup \{e_ie_j,e_i'e_j'\mid 1\leqslant i< j \leqslant m\}\cup\{ e_ie_j' \mid  i\neq j \}\cup\{e_ie_i' + e_{i+1}e'_{i+1}\mid 1\leqslant i\leqslant m-1  \}$$
is  a basis of $c(A)\cap \Skew(\C,\underline{\sigma})$.  
To complete the proof, we show that these basis elements lie in $\Alt(\C,\underline{\sigma})$. 

Pick $i,j\in\{1,\dots,m\}$ with $i\neq j$. 
As $m\geqslant 3$, 
there exists some $k\in\{1,\ldots, m\}\setminus\{i,j\}$.  As 
$e_ke_k' + e_k'e_k=1$, and  both $e_k$ and $e_k'$ commute with $e_i$ and $e_j$, and $e_i$ and $e_j$ commute, we have that 
\begin{eqnarray*}
 e_ie_j &=& e_ie_j\cdot 1 = e_ie_j(e_ke_k' + e_k'e_k) = e_ie_je_ke_k' + e_ie_je_k'e_k= e_ie_j e_ke_k'+e_k'e_ke_je_i\\
&=&  e_ie_j e_ke_k' + \underline{\sigma}(e_ie_j e_ke_k')\in\Alt(\C,\underline\sigma)\,.
\end{eqnarray*}

A similar argument shows that  $e_i'e_j'$ and $e_ie_j' \in \Alt(\C,\underline{\sigma})$.
Consider now 
\begin{eqnarray*}
e_ie_i' + e_je_j'
&=& e_ie_i'(e_ke_k'+e_k'e_k) + e_je_j'(e_ke_k'+e_k'e_k)\\
&=& e_ie_i' e_ke_k' + e_k'e_ke_ie_i'  + e_je_j' e_ke_k' + e_k'e_ke_je_j'\,.\\
\end{eqnarray*}
Using $e_ie_i' +e_i'e_i=1=e_je_j'+e_j'e_j$, we get that 
\begin{eqnarray*} e_ie_i' + e_je_j'  &=& e_ie_i' e_ke_k' + e_k'e_k(1+e_i'e_i)  + e_je_j' e_ke_k' + e_k'e_k(1+e_j'e_j) \\&=&   e_ie_i' e_ke_k' + e_k'e_ke_i'e_i  + e_je_j' e_ke_k' + e_k'e_ke_j'e_j
\\
&=& (e_ie_i' e_ke_k'  +  e_je_j' e_ke_k' ) + \underline{\sigma}(e_ie_i' e_ke_k'  +  e_je_j' e_ke_k' )
\,.
\end{eqnarray*}
In particular, $e_ie_i' + e_{i+1}e'_{i+1}\in \Alt(\C,\underline{\sigma})$ for all $1\leqslant i\leqslant m-1$, and this finishes the proof. 
\end{proof}

Using this proposition, we get : 
\begin{defi}\label{semitrace.def}
Let $(A,\sigma,f)$ be an algebra with quadratic pair, of degree $2m\geq 8$, with $m$ even and further $m\equiv 0\mod 4$ if  the characteristic of $F$ is different from $2$. 
Given $\lambda\in A$ with $\Trd_A(\lambda)=1$, the semi-trace 
\[\underline f:\, s\in\Sym(\C,\underline{\sigma})\mapsto \Trd_\C(c(\lambda)s)\] does not depend on $\lambda$. It is called the {canonical semi-trace on $(\C,\underline{\sigma})$}. We refer to the pair $(\underline{\sigma}, \underline{f})$ as the canonical quadratic pair on $\C=\C(A,\sigma,f)$.
\end{defi}

\begin{remark} \label{splitcentre}
(1) Since the reduced trace is a nonzero linear form, there exists $\lambda\in A$ such that $\Trd_A(\lambda)=1$. 

(2) Assume the characteristic of $F$ is prime to the degree of $A$, and let  $(A,\sigma)$ be an algebra with orthogonal involution. Then $\frac{1}{\deg(A)}\in A$ has reduced trace $1$, and its image in $\C(A,\sigma)$ is $c\left(\frac{1}{\deg(A)}\right)=f\left(\frac{1}{\deg(A)}\right)=\frac{1}{2}$ by~\cite[(5.6) \& (8.7)]{KMRT}. 
So $f$ is half the reduced trace of $\C$, as prescribed in this case. 

(3) This definition gives a semi-trace on the two components of the Clifford algebra  when the involution has trivial discriminant. In characteristic different from $2$, this is clear, since the involution is orthogonal on each component. Assume the characteristic of $F$ is $2$. As explained in~\cite[(8.12)]{KMRT}, since $m$ is even, the canonical involution restricts to a symplectic involution on each component. More precisely, let $\ell\in A$ be an element that gives the semi-trace $f$. As explained in the proof of~\cite[(8.28)]{KMRT}, the centre of $C(A,\sigma,f)$ is $F[c(\ell)]$. If $\textrm{disc}(\sigma,f)$ is trivial, then
$c(\ell)^2+c(\ell)=u^2+u$ for some $u\in F$. It follows that 
 $\C=\C(A,\sigma,f)$ decomposes into two components, $\C\simeq\C^+\times \C^-$, where $\C^+= \C\cdot (c(\ell)+u)$ and $\C^-=\C\cdot(c(\ell)+u+1)$. 
 By~\cite[(5.6)\&(8.16)]{KMRT}, we have $\underline\sigma(c(\ell))=c(\sigma(\ell))=c(\ell)+c(1)=c(\ell)+m$. 
  Since $m$ is even, this shows the canonical involution on $\C$ restricts to involutions, denoted $\sigma^+$ and $\sigma^-$ respectively, on both components, which are symplectic. 
 We get a pair of canonical semi-traces, respectively denoted by $f^+$ and $f^-$, and determined by $c(\lambda)^+=c(\lambda).(c(\ell)+u)$ and $c(\lambda)^-=c(\lambda).(c(\ell)+u+1)$. 

\end{remark}

The next proposition provides some evidence that the quadratic pair we have just defined is part of the structure of the Clifford algebra. Let $\theta:\,(A,\sigma,f)\rightarrow (B,\tau,g)$ be an isomorphism of algebras with quadratic pairs. It follows from Definition~\cite[(8.7)]{KMRT} that $\theta$ induces an isomorphism $\C(\theta):\,\C(A,\sigma,f)\rightarrow \C(B,\tau,g)$, satisfying 
$\C(\theta)(c_A(a))=c_B(\theta(a))$ for all $a\in A$, where $c_A$ (respectively $c_B)$ denotes the canonical map $c_A:\, A\rightarrow \C(A,\sigma,f)$ (respectively $c_B:\, B\rightarrow \C(B,\tau,g)$). Moreover, one may easily check $\C(\theta)$ preserves the canonical involutions. We claim it is an isomorphism of algebras with quadratic pairs, that is 

\begin{prop}\label{aut.prop}
Every isomorphism $\theta:\,(A,\sigma,f)\rightarrow (B,\tau,g)$ induces an isomorphism $\C(\theta):\,\bigl(\C(A,\sigma,f),\underline\sigma,\underline f\bigr)\rightarrow \bigl(\C(B,\tau,g),\underline\tau,\underline g\bigr)$. 
\end{prop}
\begin{proof}
It only remains to check that $C(\theta)$ preserves the semi-trace. 
Let $s$ be a symmetric element, $s\in\Sym\bigl(\C(A,\sigma,f),\underline\sigma\bigr)$, so that $\C(\theta)(s)\in\Sym\bigl(\C(B,\tau,g),\underline\tau\bigr)$. 
We have to prove $\underline g(\C(\theta)(s))=\underline f (s)$. 
Pick $\lambda\in A$ with $\Trd_A(\lambda)=1$. 
Since $\theta$ is an isomorphism, we have $\Trd_B(\theta(\lambda))=\Trd_A(\lambda)=1$. 
Therefore, by \cref{semitrace.def}, 
\[\underline g(\C(\theta)(s))=\Trd_{\C(B,\tau,g)}(c_B(\theta(\lambda))\C(\theta)(s))=\Trd_{\C(B,\tau,g)}(\C(\theta)(c_A(\lambda)s).\]
Since $\C(\theta)$ is an isomorphism, again it preserves the reduced trace, and we get 
\[\underline g(\C(\theta)(s))=\Trd_{\C(A,\sigma,f)}(c_A(\lambda)s)=\underline{f}(s),\]
as required. 
\end{proof}

\subsection{Explicit description in the split case}
\label{split.section}
Let $(V,q)$ be a nonsingular quadratic space of dimension $2m$, with polar form $b$. We assume that $m$ is even, and 
further that $m\equiv 0 \mod 4$ if the characteristic of $F$ is different from $2$,
so that the canonical involution $\underline{\sigma_q}$ of $\C_0(V,q)$ is of orthogonal type in characteristic different from $2$, and of symplectic type otherwise. Since $q$ is nonsingular, we may find a pair of vectors $(e,e')$ such that $b_q(e,e')=1$. 
Let $u=ee'$ be the corresponding element in $\C_0(V,q)$. We have $u+\underline{\sigma_q}(u)=ee'+e'e=b_q(e,e')=1$. Therefore, this element $u$ defines a semi-trace on $(\C_0(V,q),\underline{\sigma_q})$, which we denote by $f_{e,e'}$. We claim it coincides with the canonical semi-trace of $\bigl(\C(\Ad_q),\underline{\sigma_q})$ under the canonical identification provided in~\cite[(8.8)]{KMRT}. More precisely, we have : 
\begin{prop}
\label{split.prop}
Let $(V,q)$ be a nonsingular quadratic space of dimension $2m\geq 8$, with $m$ even and 
assume further that $m\equiv 0 \mod 4$ if the characteristic of $F$ is different from $2$. 
The standard identification $\varphi_q:\,V\otimes V\rightarrow \End_F(V)$ induces an isomorphism of algebras with quadratic pairs 
\[\bigl(\C_0(V,q),\underline{\sigma_q},f_{e,e'}\bigr)\simeq \bigl(\C(\Ad_q),\underline{\sigma_q},\underline{f_q}\bigl).\]
\end{prop}
\begin{proof}
In view of~\cite[(8.8)]{KMRT}, it only remains to identify the semi-traces. Denote by $P$ the plane generated by $e$ and $e'$ in $V$; since $q$ restricts to a nonsingular form on $P$, by~\cite[(7.22)]{Elman:2008}, we have $V=P\perp P^\perp$. 
Recall from~\cite[(5.10)]{KMRT} that $\varphi_q(e\otimes e')$ maps $x\in V$ to $e\,b_q(e',x)$. Hence it vanishes on $e'$ and $P^\perp$, and maps $e$ to itself. Therefore, $\varphi_q(e\otimes e')\in \End_F(V)$ has trace $1$. By \cref{semitrace.def}, the canonical semi-trace of $(\C(\Ad_q),\underline{\sigma_q})$ is determined by the element $c(\varphi_q(e\otimes e'))$; the corresponding element in $\C_0(q)$ is $u=ee'$, and this proves the proposition. 
\end{proof}

\begin{remark}
\label{split.rem}
It follows from this proposition that if $V$ has dimension $2m\geq 8$ with $m$ even
and 
further  $m\equiv 0 \mod 4$ if the characteristic of $F$ is different from $2$, then the semi-trace $f_{e,e'}$ on $(\C_0(V,q),\underline{\sigma_q})$ does not depend on the choice of the pair $(e,e')$ such that $b_q(e,e')=1$. 
This is obvious if the characteristic of $F$ is different from $2$, as in this case the semi-trace is unique. 

If $F$ has characteristic $2$, this can be directly checked as follows. 
Consider two such pairs $(e,e')$ and $(g,g')$, and let $u=ee'$ and $v=gg'$ be the corresponding elements in $\C_0(q)$. We need to prove that $u$ and $v$ differ by an alternating  element of $(\C_0(V,q),\underline{\sigma_q})$. Let $P$ and $Q$ be the planes respectively generated by $(e,e')$ and $(g,g')$. The polar form $b_q$ is nondegenerate on both planes. 
 We claim there exists a third plane $R$ over which $b_q$ is nondegenerate, and which is orthogonal to $P$ and $Q$. Indeed, by~\cite[Prop 1.6]{Elman:2008}, the form $b_q$ also is nondegenerate on the orthogonal $P^\perp$ of the plane $P$, which has dimension $2m-2$. Besides, $P^\perp\cap Q^\perp$ is a subspace of $P^\perp$ of dimension at least $2m-4$. Since $m\geq 4$, $P^\perp\cap Q^\perp$ has dimension strictly larger than half the dimension of $P^\perp$. Therefore, $b_q$ cannot be identically $0$ on $P^\perp\cap Q^\perp$, which proves the existence of $R$. 
Let $(h,h')$ be a symplectic base of $R$, and let $w=hh'\in\C_0(V,q)$. We have  
${\underline{\sigma_q}}(u)=u+1$, ${\underline{\sigma_q}}(v)=v+1$ and ${\underline{\sigma_q}}(w)=w+1$. Moreover, $w$ commutes with $u$ and $v$.
It follows that  
$u=v+(u+v)w+{\underline{\sigma_q}}\bigl((u+v)w\bigr)$. Hence $u$ and $v$ differ by an alternating element, so they define the same semi-trace. 
\end{remark} 

Let us now pick an explicit presentation of the quadratic form $q$; we get the following: 
\begin{prop}
\label{splitdec.prop} Assume $F$ is of characteristic $2$ and $m$ is even. 
If \[q=[a_1,b_1]\perp\dots\perp[a_m,b_m],\mbox{ then }\]  
\[(\C(\Ad_q),\underline{\sigma_q},\underline{f_q})\simeq \bigl(Q_1\otimes\dots\otimes Q_{m-1},\can\otimes\dots\otimes\can,f_\otimes\bigr)\otimes _F K,\]
where $Q_i=[a_ib_i,a_ia_m)$, $\can$ stands for the canonical involution, $f_\otimes$ is the canonical semi-trace associated to this tensor product (see \cref{canprod}), and $K$ is the quadratic \'etale extension of $F$ generated by $\Delta(q)\in F/\wp(F)$. 
\end{prop}

\begin{remark} Let $q$ be a non-degenerate quadratic form of even dimension over a field of characteristic different from $2$, and consider an orthogonal basis $(e_1,e_2,\dots, e_{2m})$ of the underlying vector space. A direct computation shows that the elements 
\[\left\{\begin{array}{l}i_1=e_1e_2\\ j_1=e_1e_3\\\end{array}\right.,\ \left\{\begin{array}{l} i_2=e_1e_2e_3e_4\\ j_2=e_1e_2e_3e_5\\\end{array}\right.,\ \dots,\  \left\{\begin{array}{l}i_{m-1}=e_1\dots e_{2m-3}e_{2m-2}\\ j_{m-1}=e_1\dots e_{2m-3}e_{2m-1}\\\end{array}\right.\] 
generate pairwise commuting $F$-quaternion algebras in $\C_0(q)$ that are stable under the canonical involution. Hence, $\C_0(q)$ is isomorphic to the tensor product of those quaternion algebras, extended from $F$ to the center $K=F[e_1\dots e_{2m}]$. It follows that $(\C_0(q), \underline{\ad_q})$ is totally decomposable as an algebra with involution. 
The proposition above extends this result. 
In particular, in characteristic 2, for a nonsingular quadratic form $q$, we have that $(\C(\Ad_q),\underline{\sigma_q},\underline{f_q})$ is a totally decomposable algebra with quadratic pair; indeed, 
\[\bigl(Q_1\otimes\dots\otimes Q_{m-1},\can\otimes\dots\otimes\can,f_\otimes\bigr)\otimes _F K\simeq \bigl({Q_1}_{K}\otimes\dots\otimes {Q_{m-1}}_{K},\can\otimes\dots\otimes\can,f_\otimes\bigr). \]
In fact, we have more, namely that $\C_0(q)$, endowed with its canonical quadratic pair, has a totally decomposable descent to $F$. 

\end{remark}
\begin{proof} 
By \cref{split.prop} and \cref{ex:cliffex} we already know that the two algebras with involution are isomorphic, and we need to check that the semi-trace $f_\otimes$ on $\bigl({Q_1}_{K}\otimes\dots\otimes {Q_{m-1}}_{K},\can\otimes\dots\otimes\can)$ corresponds to $f_{e_1,e_1'}$ on $(\C_0(V,q),\underline{\sigma_q})$. Let $(B,\tau)=\otimes_{i=2}^{m-1} ({Q_i}_{K},\can)$. 
As explained in \cref{ex:cliffex}, $u_1=e_1e_1'\in\C_0(V,q)$ is equal to $u_1\otimes 1\in {Q_1}_{K}\otimes B$. 
Therefore, $f_{e_1,e_1'}$ on $(\C_0(V,q),\underline{\sigma_q})$ is the semi-trace induced by $f_{u_1}$ on $({Q_1}_{K},\can)$ as in~\cref{prod.section}, which coincides with $f_\otimes$ by~\cref{part}. 
\end{proof} 

\begin{cor}\label{pfister.cor}
Let $\pi$ be a $3$-fold Pfister form over $F$. We have 
$$(\C(\Ad_\pi),\underline{\sigma_\pi}, \underline{f_\pi})\simeq \mathrm{Ad}_\pi\times \mathrm{Ad}_\pi\,.$$
\end{cor}

\begin{proof}
 For fields of characteristic different from $2$, this follows directly from \cite[(35.1)]{KMRT} and the uniqueness of the semi-trace. Assume now that the characteristic of $F$ is $2$.
Let $\pi= \pfr{a,b,c}$.
Using the isometry $x[1,y]\simeq [x,x^{-1}y]$ for $x\in F^\times$ and $y\in F$, we obtain that 
$$\pi \simeq     [a,a^{-1}c]  \perp   [b,b^{-1}c]  \perp [ab,(ab)^{-1}c]  \perp [1,c]\,.  $$
Hence, by \cref{splitdec.prop}, we have 
$$(\C(\Ad_\pi),\underline{\sigma_\pi}, \underline{f_\pi})=  \bigl([c, a)\otimes[c,b)\otimes[c,ab),\can\otimes\can \otimes \can, f_\otimes\bigl)\otimes_F (F\times F)\,.$$

On the other hand,  $$\Ad_\pi\simeq\Ad_{\pff{a}}\otimes \Ad_{\pfr{b,c}}\,.$$ 
Since $\pfr{b,c}$ is the norm form of the quaternion algebra $[b,c)$, using~\cite[(5.5)]{dolphin:quadpairs} and~\cite[(2.9)]{dolphin:conic} we get 
$$\Ad_\pi\simeq  ([0,a),\tau)\otimes ([c,b),\can) \otimes ([c,b),\can, f)\,, $$
where  for the quaternion basis $(1,u,v,w)$ of $[0,a)$ the orthogonal  involution  $\tau$  is characterised by $\tau(u)=u$ and $\tau(v)=v$ and $f$ is any semi-trace on $([c,b),\can)$.
Finally, for a particular choice of the  the semi-trace $f$, the isomorphism from \cite[(5.5)]{dolphin:totdecomp} gives us 
 $$([0,a),\tau)\otimes  ([c,b),\can, f) \simeq  ([c,a)\otimes [c,ab),\can\otimes \can, f_\otimes)\,. $$
This finishes the proof by \cref{part}. 
\end{proof} 

We also prove the following extension of \cite[(8.5)]{KMRT}.

\begin{cor}\label{isohypo}
Let $q$ be a nonsingular  quadratic form over $F$ of even dimension $2m\geq 8$ with $m$ even. 
If $q$ is isotropic then  $(\C_0(\Ad_q),\underline{\sigma_q},\underline{f_q})$ is hyperbolic.
\end{cor}
\begin{proof} For the case where $F$ is of characteristic different from $2$, see  \cite[(8.5)]{KMRT}. We now assume that $F$ is of characteristic $2$. 
As $q$ is isotropic we have $q\simeq \HH\perp q'\simeq [1,0]\perp q'$ for some nonsingular quadratic form $q'$ over $F$.
Hence we may assume $a_1=1$ and $b_1=0$ in \cref{splitdec.prop}, and we get 
$$(\C_0(\Ad_q),\underline{\sigma_q}, \underline{f_q})\simeq ([0,c),\can,f)\otimes (B,\tau)\otimes_F K$$
for some $c\in F^\times$, some arbitrary choice of a semi-trace $f$ on $([c,0),\can)$, and some  $F$-algebra with  
symplectic involution $(B,\tau)$ (see \cref{differentdecomp}). Since $[0,c)$ is a split algebra, we may choose $f$ so that 
$([0,c),\can,f)$ is the adjoint of a hyperbolic plane, and it follows that the Clifford algebra $(\C_0(\Ad_q),\underline{\sigma_q},\underline{f_q})$ is hyperbolic. 
\end{proof}

\section{Triality} 
\label{sec:tri}

\subsection{An action of $A_3$ on $\PGO^+(n)$}
\label{sub:action}

Let $\co$ be a Cayley algebra, and denote by $\star$ its para-Cayley product, defined by $x\star y=\bar x\bar y$, see~\cite[\S 34.A]{KMRT}. The algebra $(\co,\star,n)$ is a symmetric composition algebra, where $n$ is the norm form of $\co$. In particular, the norm form is multiplicative, that is 
$n(x\star y)=n(x)n(y)$ for all $x,y\in\co$. Moreover, 
 we have 
\begin{equation}
\label{comp.eq} \forall x,y\in \co,\ \ \ x\star(y\star x)=n(x) y=(x\star y)\star x
\end{equation} see~\cite[(34.1)]{KMRT}.
In this section, we describe an action of $A_3$ on $\PGO^+(n)$, induced by the algebra structure of $\co$, and which will be used to study algebras with quadratic pairs of degree $8$ and trivial discriminant. Similar computations were recently made by Alsaody and Gille \cite[\S 4]{AG}, where they work over more general base (a unital commutative ring), and consider triples of isometries, while we consider triples of similitudes. Our approach follows \cite[\S 35]{KMRT}, with the additional ingredient that the Clifford algebra is induced with a canonical quadratic pair rather than just an involution. 

The main result in this section is the following : 
\begin{prop} 
\label{triple}
Let $t$ be a proper similitude of $(\co,n)$ with multiplier $\mu(t)$. There exist proper similitudes $(t^+,t^-)$ of $(\co,n)$ such that 
\begin{enumerate}
\item[(a)] $ t^+(x\star y)=\mu(t^+)\,t(x)\star t^-(y);$
\item[(b)] $ t(x\star y)=\mu(t)\,t^-(x)\star t^+(y);$
\item[(c)] $t^-(x\star y)=\mu(t^-)\,t^+(x)\star t(y).$
\end{enumerate} 
The pair $(t^+,t^-)$ is uniquely defined up to a factor $(\lambda^{-1},\lambda)$ for some $\lambda\in F^\times$ and the multipliers satisfy $\mu(t^+)\mu(t)\mu(t^-)=1$. 
\end{prop}

From this, we derive an action of the alternating group $A_3\simeq \mz/3$ on $\PGO^+(n)$ as follows. 
Given a proper similitude $t$ of $(\co,n)$ we denote by $[t]$ its class in $\PGO^+(n)$. It is clear from the relations above that $[t^{++}]=[t^-]$ and $[t^{+++}]=[t]$. Hence, we get :
\begin{cor}
The assignments $[t]\mapsto \theta^+([t])= [t^+]$ and $[t]\mapsto \theta^-([t])=[t^-]$ define an action of $A_3$ on $\PGO^+(n)$.  
\end{cor} 

\begin{remark}
In \cite{AG}, Alsaody and Gille give an explicit description of the spin group $\Spin(n)$ with its trialitarian action (see \cite[Lem.~3.3 and Thm.~3.9]{AG}). Their description is in terms of so-called related triples, which correspond to triples as in our \cref{triple}, except that $t$, $t^+$ and $t^-$ are isometries rather than similitudes. It is clear from their work that the action described in this section is the induced trialitarian action on $\PGO^+(n)$. The explicit description we provide could also be deduced from their results by fppf descent. 
\end{remark}

The remainder of this section outlines the proof of Proposition~\ref{triple}. The argument is mostly borrowed from~\cite[\S 34]{KMRT}, except for Lemma~\ref{triple.lem} which adds the canonical quadratic pair to the picture.

For all $x\in \co$, we denote by $r_x$ and $\ell_x$ the endomorphisms of $\co$ defined by $r_x(y)=y\star x$ and $\ell_x(y)=x\star y$. We first prove : 
\begin{lem} Let $t$ be a proper similitude of $(\co,n)$. The map 
\[\psi_t:\,\co\rightarrow \End_F(\co\oplus\co)\] defined by 
\[\psi_t(x)=\left(\begin{array}{cc}
0&\ell_{t(x)}\\\mu(t)^{-1}r_{t(x)}&0\\\end{array}\right)\]
induces isomorphisms $\C(n)\simeq \End_F(\co\oplus\co)$ and 
$\C_0(n)\simeq \End_F(\co)\times\End_F(\co)$, which we denote by $\Psi_t$. 
\end{lem} 
\begin{proof}
A direct computation shows that for all $x,y\in\co$, we have 
\begin{equation}
\label{PsiProd}
\psi_t(x)\psi_t(y)=\mu(t)^{-1}
 \left(\begin{array}{cc}
\ell_{t(x)}\circ r_{t(y)}&0\\0&r_{t(x)}\circ\ell_{t(y)}\\\end{array}\right). 
\end{equation}
In view of~\eqref{comp.eq}, it follows that $\psi_t(x)^2=\mu(t)^{-1}n(t(x))=n(x)$. 
By the universal property of the Clifford algebra, we get a non trivial map $\Psi_t:\,\C(n)\rightarrow\End_F(\co\oplus\co)$, which is an isomorphism since both algebras are simple and of the same dimension. The computation of $\psi_t(x)\psi_t(y)$ above shows that this isomorphism sends $\C_0(n)$ to $\End_F(\co)\times\End_F(\co)$, which embeds diagonally in $\End_F(\co\oplus\co)$. 
\end{proof}
Assume now that $t=\id$ and consider the corresponding isomorphisms, denoted by $\Psi_1$. 
The next lemma is a refined version of~\cite[(35.1)]{KMRT} (see also~\cite[Prop. 3.10]{AG}):
\begin{lem} 
\label{triple.lem}
The isomorphism $\Psi_1$ restricted to the even Clifford algebra is an isomorphim of algebras with quadratic pairs 
\[\Psi_1:\,(\C_0(n),\underline{\sigma_n},\underline{f_n}\bigr)\rightarrow \Ad_n\times\Ad_n,\] 
where $(\underline{\sigma_n},\underline{f_n})$ stands for the canonical quadratic pair on $\C_0(n)$.\end{lem}
\begin{remark}
Since all $3$-fold Pfister forms are norm forms of some Cayley algebras, this lemma gives a new proof of \cref{pfister.cor}. 
\end{remark}
\begin{proof}
We already know that $\Psi_1$ is an isomorphism of algebras, and one may check it preserves the involution as in~\cite[Prop. 3.10]{AG}. It only remains to prove that it is compatible with the semi-traces. Therefore, we may assume that the characteristic of $F$ is $2$. Let $(e_i,e_i')_{1\leq i\leq 4}$ be a symplectic basis of $(\co,n)$. By \cref{split.prop}, the canonical semi-trace $\underline{f_n}$ on $\C_0(n)$ 
is determined by the element $e_1e_1'\in\C_0(n)$. Under the isomorphism $\Psi_1$, it corresponds to the semi-trace determined by the element 
$\Psi_1(e_1e_1')=\psi_1(e_1)\psi_1(e_1')=\left(\begin{array}{cc} 
\ell_{e_1}\circ r_{e_1'}&0\\
0&r_{e_1}\circ\ell_{e_1'}\\
\end{array}\right)$. 
Hence, we have to prove that the elements $\ell_{e_1}\circ r_{e_1'}$ and $r_{e_1}\circ\ell_{e_1'}\in\End_F(\co)$ determine the semi-trace $f_n$ associated to the norm form $n$. 
By~\cite[(5.11)]{KMRT}, this means we have to check that $\forall v\in\co$, 
\[\Trd_{\End_F(\co)}\bigl(\ell_{e_1}\circ r_{e_1'}\circ\varphi_n(v\otimes v)\bigr)=n(v),\] and similarly for the endomorphisms $r_{e_1}\circ\ell_{e_1'}\circ\varphi_n(v\otimes v)$, where $\varphi_n$ is the standard identification $\co\otimes\co\simeq \End_F(\co)$ defined in~\cite[(5.2)]{KMRT}. 
Since the value of a semi-trace is determined on symmetrised elements, it is enough to prove this equality when $v$ is one of the basis elements $(e_i,e_i')_{1\leq i\leq 4}$, 
see the proof of~\cite[(5.11)]{KMRT}. 
For all $x\in\co$, we have $\ell_{e_1}\circ r_{e_1'}\circ\varphi_n(e_i\otimes e_i)(x)=\bigl(e_1\star( e_i\star e_1')\bigr)\,b_n(e_i,x)$. 
This endomorphism maps all elements of the basis to $0$, except for $e_i'$. Hence its trace is the coordinate of $e_1\star( e_i\star e'_1)$ on $e_i'$, that is 
$b( e_1\star( e_i\star e'_1),e_i)$. By~\cite[\S 34]{KMRT}, we get 
\[b( e_1\star (e_i\star e'_1),e_i)=b(e_i,e_1\star (e_i\star e'_1))=b(e_i\star e_1,e_i\star e_1')=n(e_i)b(e_1,e'_1)=n(e_i),\]
as required. 
A similar computation shows the equality also holds for $e_i'\otimes e_i'$, and for the endomorphism $r_{e_1}\circ\ell_{e_1'}\circ\varphi_n(v\otimes v)$ instead of $\ell_{e_1}\circ r_{e_1'}\circ\varphi_n(v\otimes v)$, so the lemma is proved. 
\end{proof}

With this in hand, we may now prove Proposition~\ref{triple} as follows. Given a proper similitude $t$ of $(\co,n)$, consider the isomorphisms $\Psi_1$ and $\Psi_t$. 
By the Skolem-Noether theorem, there exists an invertible element  $S\in\End_F(\co\oplus\co)\simeq M_2(\End_F(\co))$ such that the following diagram commutes : 
\begin{equation*}
\xymatrix{%
\C(n)\ar[d]_{\Psi_1}\ar[dr]^{\Psi_t}&\\
\End_F(\co\oplus\co)\ar[r]_{\Int(S)}&\End_F(\co\oplus\co).}
\end{equation*}
Restriction to the even part of all three algebras shows that $\Int(S)$ preserves $\End_F(\co)\times \End_F(\co)\subset M_2(\End_F(\co))$, so that \[S=\begin{pmatrix}s_0&0\\0&s_2\\\end{pmatrix}\mbox{ for some }s_0,s_2\in\End_F(\co).\]
Recall that $t$ also induces an isomorphism $\C_0(t):\,\C_0(n)\rightarrow \C_0(n)$ which preserves the canonical quadratic pair by~\cref{aut.prop}. We claim that the following diagram is commutative : 
\begin{equation*}
\xymatrix{%
\C_0(n)\ar[d]_{\Psi_1}\ar[dr]^{\Psi_t}\ar[r]^{\C_0(t)}&\C_0(n)\ar[d]^{\Psi_1}\\
\End_F(\co)\times\End_F(\co)\ar[r]_{\Int(S)}&\End_F(\co)\times\End_F(\co).}
\end{equation*}
Indeed, the lower triangle is obtained from the previous commutative diagram by restriction to the even part. Since $\C_0(t)(xy)=\mu(t)^{-1}t(x)t(y)$ for all $x,y\in\co$ (see~\cite[(13.1)]{KMRT}), the upper triangle also commutes by a direct computation using \eqref{PsiProd}. In view of \cref{triple.lem} and \cref{aut.prop}, the automorphism $\Int(S)$ preserves the quadratic pair $\ad_n\times \ad_n$, so that $s_0$ and $s_1$ are similitudes of $(\co,n)$. 

Finally, since $\Psi_t=\Int(S)\circ \Psi_1$, we have for all $x\in\co$, $\psi_t(x)=S\psi_1(x)S^{-1}$. 
Hence, we get
$\mu(t)^{-1}r_{t(x)}=s_2 r_x s_0^{-1}$ and $\ell_t(x)=s_0\ell_x s_2^{-1},$
so that for all $y\in \co$, 
$\mu(t)^{-1} s_0(y)\star t(x)=s_2(y\star x)$ and  $t(x)\star s_2(y)=s_0(x\star y).$
Applying the norm $n$ to the second equality, we get $\mu(t)\mu(s_2)=\mu(s_0)$. 
Hence, the similitudes $t^+=\mu(s_0)^{-1}s_0$ and $t^-=s_2$ have 
$\mu(t)\mu(t^+)\mu(t^-)=1$ and satisfy equations (a) and (c) in \cref{triple}. 
Equation (b) follows from (a) and (c), as explained in~\cite[p.484]{KMRT}. 
Since 
\begin{equation}
\label{Int(S)}
\Psi_1\circ \C_0(t)\circ\Psi_1^{-1}=\Int\begin{pmatrix} 
t^+&0\\
0&t^-\\
\end{pmatrix},
\end{equation}the pair $(t^+,t^-)$ is unique up to a pair of scalars. The condition on the multipliers guarantees it actually is unique up to $(\lambda^{-1}, \lambda)$, for some $\lambda\in F^\times$. 
 It only remains to prove that $t^+$ and $t^-$ are proper; as explained in~\cite[\S 35.B]{KMRT}, if one of them was improper, it would satisfy a relation similar to~\cite[(35.4)(5)(6)]{KMRT} instead of relations (b) and (c) above. This concludes the proof. 

\begin{remark}
\label{triality.rem}
It follows from the proof that, given a proper similitude $t\in \PGO_8^+$, we have  $\theta^+([t])=[t^+]$ and $\theta^-([t])=[t^-]$, where $[t^+]$ and $[t^-]$ are characterized by equation~\eqref{Int(S)} above, and $\Psi_1$ is as in \cref{triple.lem}. 
\end{remark}

\subsection{Trialitarian triples}
A trialitarian triple over $F$ is an ordered triple of degree $8$ central simple algebras with quadratic pairs over $F$, 
\[\bigl((A,\sigma_A,f_A);(B,\sigma_B,f_B);(C,\sigma_C,f_C)\bigr),\] such that there exists an isomorphism 
\[\alpha_A:\,\bigl(\C(A,\sigma_A,f_A),\underline{\sigma_A},\underline {f_A}\bigr)\rightarrow (B,\sigma_B,f_B)\times (C,\sigma_C,f_C).\] 
Two such triples, denoted by $(A,B,C)$ and $(A',B',C')$ for short, are called isomorphic if there exists isomorphism of algebras with quadratic pairs \[\phi_A:\,(A,\sigma_A,f_A)\rightarrow (A',\sigma_{A'},f_{A'}),\]
\[\phi_B:\,(B,\sigma_B,f_B)\rightarrow (B',\sigma_{B'},f_{B'}),\]
\[\mbox{ and }\phi_C:\,(C,\sigma_C,f_C)\rightarrow (C',\sigma_{C'},f_{C'}),\]
and $\alpha_A$ and $\alpha_{A'}$ as above such that 
the following diagram commutes 
\begin{equation*}
\xymatrix{%
{\bigl(\C(A,\sigma_A,f_A),\underline{\sigma_A},\underline {f_A}\bigr)}\ar[r]^{\alpha_A}\ar[d]_{C(\phi_A)}&{(B,\sigma_B,f_B)\times (C,\sigma_C,f_C)}\ar[d]^{\phi_B\times\phi_C}\\
{\bigl(\C(A',\sigma_{A'},f_{A'}),\underline{\sigma_{A'}},\underline {f_{A'}}\bigr)}\ar[r] ^{\alpha_{A'}}&(B',\sigma_{B'},f_{B'})\times ({C'},\sigma_{C'},f_{C'}).}
\end{equation*}

\begin{remark}
If $(A,B,C)$ is a trialitarian triple, it follows from the definition that $\C(A,\sigma_A,f_A)$ has centre $F\times F$, hence the quadratic pair $(\sigma_A, f_A)$ has trivial discriminant (see~\cite[(7.7) \& (8.28)]{KMRT}). 
\end{remark} 

\begin{ex} 
Assume  $F$ is of characteristic $2$.
Let $q$ be an $8$-dimensional quadratic form with trivial Arf invariant, and let $(A,\sigma_A,f_A)=\Ad_q$. 
Pick a presentation  \[q=[a_1,b_1]\perp [a_2,b_2]\perp[a_3,b_3]\perp[a_4,b_4].\]  
By~\cref{splitdec.prop}, $(A,B,B)$ is a trialitarian triple, where $B$ stands for 
\[(B,\sigma_B,f_B)=([a_1b_1,a_1a_4)\otimes[a_2b_2,a_2a_4)\otimes[a_3b_3,a_3a_4),\can\otimes\can\otimes\can,f_\otimes\bigl).\]
\end{ex}
Hence if the algebra $A$ is split in a trialitarian triple $(A,B,C)$, then $B$ and $C$ are isomorphic. The converse also holds, as we now explain: 
\begin{lem}
\label{split.lem}
Let $(A,B,C)$ be a trialitarian triple. The following assertions are equivalent: 
\begin{enumerate}
\item The algebra $A$ is split; 
\item The triples $(A,B,C)$ and $(A,C,B)$ are isomorphic; 
\item The algebras with quadratic pairs $(B,\sigma_B,f_B)$ and $(C,\sigma_C,f_C)$ are isomorphic;
\item The algebras $B$ and $C$ are Brauer equivalent.  
\end{enumerate}
\end{lem}
\begin{proof}
Assume $A$ is split, and consider an improper isometry of the underlying quadratic space. It induces an automorphism $\phi_A$ 
of $(A,\sigma_A,f_A)$ such that $\C(\phi_A)$ acts non trivially on $F\times F$. Therefore, if $\varepsilon :B\times C\rightarrow C\times B$ denotes the switch map, defined by $\varepsilon(x,y)=(y,x)$, then $\varepsilon\circ\alpha_A\circ \C(\phi_A)\circ\alpha_A^{-1}$ is an  isomorphism $B\times C\rightarrow C\times B$ which acts trivially on $F\times F$. Hence, it is equal to $(\phi_B,\phi_C)$ for some isomorphisms of algebras with quadratic pairs 
$\phi_B:\,B\rightarrow C$ and $\phi_C:C\rightarrow B$. This shows that $(\phi_A,\phi_B,\phi_C,\alpha_A,\varepsilon\circ\alpha_A)$ defines an isomorphism of triples between $(A,B,C)$ and $(A,C,B)$. 
Assertion (3) follows from (2) by definition, and it clearly implies (4). Finally, since $A$ has degree $8$, the so-called fundamental relations given in~\cite[(9.13) and (9.14)]{KMRT} prove that (4) implies (1). 
\end{proof} 

\begin{lem}
There is a bijection between $H^1(F,PGO_8^+)$ and isomorphism classes of trialitarian triples.  
\end{lem} 

\begin{proof}
Let $n_0$ be the $8$-dimensional hyperbolic quadratic form, so that $\PGO_8^+=\PGO^+(n_0)$. 
According to~\cite[\S 29.F]{KMRT}, $H^1(F,\PGO_8^+)$ corresponds bijectively to isomorphism classes of quadruples 
$(A,\sigma_A,f_A,\varepsilon_A)$, where $(A,\sigma_A,f_A)$ is a degree $8$ algebra with quadratic pair, and $\varepsilon_A:\,Z_A\rightarrow F\times F$ is a fixed isomorphism from the centre of the Clifford algebra of $(A,\sigma_A,f_A)$ and $F\times F$, which is the centre of $\C_0(n_0)$. 
To such a quadruple, we may associate a trialitarian triple $(A,B,C)$, where $B$ and $C$ are defined by 
\[\left\{\begin{array}{l} B=\C(A,\sigma_A,f_A)\,e,\ \ \ C=\C(A,\sigma_A,f_A)(1+e), \\
\mbox{and }e=\varepsilon_A^{-1}\bigl((1,0)\bigr)\in Z_A\subset \C(A,\sigma_A,f_A).\\
\end{array}\right.\] 
Since  $A$ has degree $8$, the canonical involution $\underline{\sigma_A}$ acts trivially on $Z_A$. Hence the canonical pair $(\underline{\sigma_A},\underline{f_A})$ induces quadratic pairs $(\sigma_B,f_B)$ and $(\sigma_C,f_C)$ on each component, see~\cref{splitcentre}. Moreover, one may check that isomorphic quadruples lead to isomorphic trialitarian triples. 

Conversely, given a trialitarian triple $(A,B,C)$ pick an isomorphism $\alpha_A$ between $\C(A)$ and $B\times C$, and define the $\varepsilon_A$ to be the restriction of $\alpha_A$ to the centre $Z_A$ of the Clifford algebra of $A$. We claim that the isomorphism class of the quadruple $(A,\sigma_A,f_A,\varepsilon_A)$ does not depend on the choice of $\alpha_A$. If all such isomorphisms have the same restriction to the centre, this is clear. Assume now that there exists $\alpha_{A}^{(1)}$ and $\alpha_{A}^{(2)}$ having different restrictions. Then the composition $\alpha_{A}^{(2)}\circ(\alpha_A^{(1)})^{-1}$ is an isomorphism of the algebra with quadratic pair $B\times C$ whose restriction to the centre $F\times F$ is the non trivial automorphism. Hence, $B$ and $C$ are isomorphic, and $A$ is split by \cref{split.lem}. In this case, the algebra with quadratic pair $(A,\sigma_A,f_A)$ admits improper similitudes, and it follows that the quadruples $(A,\sigma_A,f_A,\varepsilon_A^{(1)})$ and $(A,\sigma_A,f_A,\varepsilon_A^{(2)})$ are isomorphic. 

Therefore, the set of isomorphism classes of quadruples $(A,\sigma_A,f_A,\varepsilon_A)$  and the set of isomorphism classes of trialitarian triples $(A,B,C)$ are in bijection, as required. 
\end{proof} 

\subsection{Action of $A_3$ on trialitarian triples} 

The main result of this section is the following, which extends
\cite[(42.3)]{KMRT} to characteristic $2$: 

\begin{thm}\label{thm:trirel}
The action of $A_3$ on $\PGO_8^+$ induces an action on trialitarian triples, which is given by permutations. 
In particular, if
\[\bigl(\C(A,\sigma_A,f_A),\underline{\sigma_A},\underline{f_A}\bigr)\buildrel\sim\over\longrightarrow (B,\sigma_B,f_B)\times (C,\sigma_C,f_C),\] then we also have 
 \[\bigl(\C(B,\sigma_B,f_B),\underline{\sigma_B},\underline{f_B}\bigr)\buildrel\sim\over\longrightarrow  (C,\sigma_C,f_C)\times (A,\sigma_A,f_A),\] and 
 \[\bigl(\C(C,\sigma_C,f_C),\underline{\sigma_C},\underline{f_C}\bigr)\buildrel\sim\over\longrightarrow  (A,\sigma_A,f_A)\times (B,\sigma_B,f_B).\] 
\end{thm} 
\begin{proof}
Let $(A,B,C)$ be a trialitarian triple, and fix an isomorphism \[\alpha_A:\,\bigl(\C(A,\sigma_A,f_A),\underline{\sigma_A},\underline{f_A}\bigr)\rightarrow (B,\sigma_B,f_B)\times (C,\sigma_C,f_C).\]
As above, we let $n_0$ be the $8$-dimensional hyperbolic form. 
Recall that $\Psi_1$ defined as in \cref{triple.lem} is an isomorphism 
\[\Psi_1:\,(\C_0(n_0),\underline{ \sigma_{n_0}},\underline{f_{n_0}})\rightarrow \Ad_{n_0}\times\Ad_{n_0},\]
so that $(\Ad_{n_0},\Ad_{n_0},\Ad_{n_0})$ also is a trialitarian triple. 
After scalar extension to a separable closure $F_s$ of the base field $F$, both triples are isomorphic. 
More precisely, consider an arbitrary isomorphism 
\[\phi_A:\,(\Ad_{n_0})_{F_s}\rightarrow (A,\sigma_A,f_A)_{F_s}.\] Composing $\phi_A$ with an improper similitude of $n_0$ if necessary, we may assume that the composition 
$\alpha_A\circ\C_0(\phi_A)\circ\Psi_1^{-1}$ acts trivially on $F\times F$, so that it is given by $(\phi_B,\phi_C)$ for some isomorphisms of algebras with quadratic pairs 
\[\phi_B:\,(\Ad_{n_0})_{F_s}\rightarrow (B,\sigma_B,f_B)_{F_s}\mbox{ and }\phi_C:\,(\Ad_{n_0})_{F_s}\rightarrow (C,\sigma_C,f_C)_{F_s}.\]
Hence we get an isomorphism of trialitarian triples, that is a commutative diagram
\begin{equation*}
\xymatrix{%
(\C_0(n_0),\underline{\sigma_{n_0}},\underline{f_{n_0}})_{F_s}\ar[r]^{\Psi_1}\ar[d]_{\C(\phi_A)}&(\Ad_{n_0})_{F_s}\times(\Ad_{n_0})_{F_s}\ar[d]^{\phi_B\times\phi_C}\\
{\bigl(\C(A,\sigma_{A},f_{A}),\underline{\sigma_{A}},\underline {f_{A}}\bigr)}\ar[r] ^{\alpha_{A}}&(B,\sigma_{B},f_{B})\times ({C},\sigma_{C},f_{C}).}
\end{equation*}

Identifying the automorphism group of $(\Ad_{n_0})_{F_s}$ with $\PGO_8^+(F_s)$, we get by Galois descent that the map \[a: \Gamma_F\rightarrow \PGO_8^+(F_s),\  \gamma \mapsto \phi_A^{-1}\circ{}^\gamma\phi_A\] is a $1$-cocycle whose cohomology class determines the triple $(A,B,C)$. 
Finally, from the commutative diagram above, we have 
\[\Psi_1\circ\C_0(\phi_A^{-1}\circ{}^\gamma\phi_A)\circ \Psi_1^{-1}=
\bigl(\Psi_1\C(\phi_A^{-1})\alpha_A^{-1}\bigr)\circ\bigl(\alpha_A\C(^\gamma\phi_A)\Psi_1^{-1}\bigl)=(\phi_B^{-1}\circ{}^\gamma\phi_B,\phi_C^{-1}\circ{}^\gamma\phi_C).\]

In view of the description of the trialitarian action in \cref{sub:action}, see also \cref{triality.rem}, we get that $\theta^+(a)$ and $\theta^-(a)$ coincide with the cohomology classes of the cocycles $ \gamma \mapsto \phi_B^{-1}\circ{}^\gamma\phi_B$ and $ \gamma \mapsto \phi_C^{-1}\circ{}^\gamma\phi_C$, respectively. Hence, $\theta^+(A,B,C)$ and $\theta^-(A,B,C)$ are trialitarian triples having respectively $B$ and $C$ as a first slot. 
Finally, we have $(\theta^+)^2=\theta^-$ and $\theta^-\theta^+=\id$. Applying these formulas to the triple $(A,B,C)$ we get that the second and the third slots in $\theta^+(A,B,C)$ respectively are the first slots in $\theta^-(A,B,C)$ and in $(A,B,C)$, that is $\theta^+(A,B,C)=(B,C,A)$. The same kind of argument shows $\theta^-(A,B,C)=(C,A,B)$, and this finishes the proof. 
\end{proof}

\section{Applications of Triality}\label{sec:app}

\cref{thm:trirel} above shows that the Clifford algebra, viewed as an algebra with quadratic pair, actually is a complete invariant for degree $8$ algebras with quadratic pair with trivial discriminant. 
As a first application of our main result, we now characterize totally decomposable algebras with quadratic pair in degree $8$, see \cref{thm:totdecomp}. 
The proof uses \cref{split.lem}, which describes all triples including a split algebra. 
Using direct sums of algebras with quadratic pairs, we then provide exemples of triples, in which all three slots decompose as a sum of two degree $4$ totally decomposable algebras with quadratic pair. Finally, we prove that all trialitarian triples that include two algebras of index at most $2$ are of this shape. 

\subsection{Totally decomposable quadratic pairs}
 
Using the trialitarian action described in the previous section, we may characterize totally decomposable degree $8$ algebras with quadratic pairs as follows: 
\begin{thm}\label{thm:totdecomp}
Let $(A,\sigma,f)$ be an $F$-algebra with quadratic pair with $\deg(A)=8$. Then $(A,\sigma,f)$ is totally decomposable if and only if it has trivial discriminant and its Clifford algebra has a split factor. 
\end{thm}
\begin{proof}
The case of $F$ of characteristic different from $2$ follows immediately from  \cite[(42.11)]{KMRT} and the uniqueness of the semi-trace in this case. 
 Assume now that $F$ is of characteristic $2$.  First consider the case where $(A,\sigma,f)$ has trivial discriminant, and its Clifford algebra has a split factor. 
This means $(A,\sigma,f)$ is part of a trialitarian triple $(A,B,C)$ with $B$ or $C$ split. 
By \cref{thm:trirel}, it is also part of a triple whose first slot is split, and in view of~\cref{split.lem}, we get a quadratic form $q$ such that 
\[\bigr(\C(\Ad_q),\underline{\sigma_q},\underline{f_q}\bigl)\simeq (A,\sigma,f)\times(A,\sigma,f).\]
In view of ~\cref{splitdec.prop}, this shows $(A,\sigma,f)$ is totally decomposable. 

Assume conversely that $(A,\sigma,f)$ is totally decomposable, and pick a decomposition 
\[(A,\sigma,f) = ([a_1,b_1)\otimes [a_2,b_2)\otimes [a_3,b_3), \can\otimes\can \otimes \can, f_\otimes)\,. \]
Let $q= [b_1,a_1b_1^{-1}]\perp [b_2,a_2b_2^{-1}]\perp [b_3,a_3b_3^{-1}] \perp [1,a_1+a_2+a_3]$. 
Then $q$ has trivial Arf invariant, and applying again \cref{splitdec.prop}, we get
\[(\C(\Ad_q),\underline{\sigma_q},\underline{f_q}) \simeq (A,\sigma, f)\times  (A,\sigma, f)\,.\]
Hence by~\cref{thm:trirel}, we have
\[\bigl(\C(A,\sigma,f),\underline{\sigma},\underline f\bigr)\simeq (A,\sigma,f)\times \Ad_q.\]
This proves $(A,\sigma,f)$ has trivial Arf invariant and its Clifford algebra has a split component. 
\end{proof} 

\subsection{Examples of trialitarian triples}\label{triples.sec}
In this section, we provide explicit examples of trialitarian triples, and we prove all triples that include at least two algebras of Schur index at most $2$ are of this shape, as well as all isotropic triples. 
We use the following definition, which was first introduced for algebras with involution by Dejaiffe~\cite{dejaiffe:orthsums}, and later extended to quadratic pairs in~\cite[p.379]{tignol:qfskewfield} (see also \cite[Def. 1.4]{Tignol:galcohomgps}, and~\cite[Prop. 7.4.2]{GilleLNM}). 
\begin{defi} 
The algebra with quadratic pair $(A,\sigma,f)$ is called an orthogonal sum of $(A_1,\sigma_1,f_1)$ and $(A_2,\sigma_2,f_2)$, and we write  
\[(A,\sigma,f)\in\,(A_1,\sigma_1,f_1)\boxplus(A_2,\sigma_2,f_2),\] 
if there are symmetric orthogonal idempotents $e_1$ and $e_2$ in the algebra $A$ such that for $i\in\{1,2\}$, 
\[(e_iAe_i,\sigma_{|e_iAe_i})\simeq (A_i,\sigma_i),\] so that we may identify $A_i$ with a subset of $A$, and 
\[\forall s_i\in \Sym(A_i,\sigma_i),\ f(s_i)=f_i(s_i).\] 
\end{defi} 
Note that the identification of $A_i$ with its image in $A$ is compatible with the reduced trace. More precisely, we have  \[\Trd_A(a_i)=\Trd_{A_i}(a_i)\mbox{ for all }a_i\in A_i\simeq e_iAe_i\subset A,\] see the matrix description of the orthogonal sum given in~\cite[\S2]{dejaiffe:orthsums}. 
Moreover, the direct product $(A_1,\sigma_1,f_1)\times (A_2,\sigma_2,f_2)$ embeds in $(A,\sigma,f)$, meaning there is an embedding of the direct product of algebras with involution, and the restriction of $f$ to the image of $\Sym(A_i,\sigma_i)$ coincides with $f_i$ for $i\in\{1,2\}$. 
\begin{ex}
Let $(V_1,q_1)$ and $(V_2,q_2)$ be two non singular quadratic spaces over $F$. For all $\mu\in F^\times$, we have 
\[\Ad_{q_1\perp\qform{\mu}q_2}\in \Ad_{q_1}\boxplus\Ad_{q_2}.\]
This follows directly from the description of $\Ad_q$ given in~\cite[\S 5.B]{KMRT}, and the definition above, taking $e_1$ and $e_2$ in $A=\End_F(V_1\oplus V_2)$ to be the orthogonal projections on $V_1$ and $V_2$ respectively. 
\end{ex}
As this example shows, an orthogonal sum $(A,\sigma,f)$ is not uniquely determined by its summands $(A_1,\sigma_1,f_1)$ and $(A_2,\sigma_2,f_2)$ and $(A_1,\sigma_1,f_1)\boxplus(A_2,\sigma_2,f_2)$ should be considered as a set. 

 With this in hand, we may produce examples of trialitarian triples as follows. 
Let $Q_1,Q_2,Q_3$ and $Q_4$ be quaternion algebras such that $\bigotimes_{i=1}^4Q_i$ is split.
For all $i$ and $j$ with $i\not =j$, we denote by $f_{ij}$ the semi-trace $f_\otimes$ on $\Sym(Q_i\otimes Q_j,\can\otimes\can)$ associated to the tensor product decomposition $(Q_i,\can)\otimes (Q_j,\can)$ as in \cref{canprod}. 
 We have the following 
 \begin{prop}\label{3orthsums}
 Let $(A,\sigma,f)$ be an $F$-algebra with quadratic pair such that
$$(A,\sigma,f)\in \left(Q_1\otimes Q_2,\can\otimes\can, f_{12} \right)\boxplus\left(Q_3\otimes Q_4,\can\otimes\can, f_{34} \right)\,. $$
Then $(\sigma,f)$ has trivial discriminant, and the Clifford algebra $\C(A,\sigma,f)$, with its canonical quadratic pair, is a direct product of 
\[\C^+(A,\sigma,f)\in\left(Q_1\otimes Q_3,\can\otimes\can, f_{13} \right)\boxplus \left(Q_2\otimes Q_4,\can\otimes\can, f_{24} \right),\] 
\[\mbox{ and }\ \C^-(A,\sigma,f)\in\left(Q_1\otimes Q_4,\can\otimes\can, f_{14} \right)\boxplus\left(Q_2\otimes Q_3,\can\otimes\can, f_{23} \right).\] 
\end{prop}
\begin{proof}
In characteristic different from $2$, the algebra with involution version of this result is stated and proved in~\cite[Prop.6.6]{jinv}, and the proposition follows immediately by uniqueness of the semi-trace in this case. So we may assume $F$ has characteristic $2$. The same argument as in characteristic different from $2$ applies to describe $\C(A,\sigma,f)$ with its canonical involution, and the definition above shows it only remains to check the canonical semi-trace $\underline{f}$ acts as $f_{ij}$ on each subset \[\Sym(Q_i\otimes Q_j,\can\otimes \can)\subset\Sym\bigl(\C(A,\sigma,f),\underline{\sigma}\bigr).\]

For $i\in\{1,2\}$, let $u_i$ in $Q_i$ be a quaternion such that $\bar u_i+u_i=1$. Identifying $Q_1\otimes Q_2$ to a subset of $A$ as above, we have
\[\Trd_A(u_1\otimes u_2)=\Trd_{Q_1\otimes Q_2}(u_1\otimes u_2)=\Trd_{Q_1}(u_1)\Trd_{Q_2}(u_2)=1.\] Therefore, the canonical semi-trace $\underline f$ on $\C(A,\sigma,f)$ is determined by the element $c(u_1\otimes u_2)$ in $\C(A,\sigma,f)$ by \cref{semitrace.def}. 

Recall from~\cite[(15.12)]{KMRT} that the Clifford algebra of $(Q_i\otimes Q_j,\can\otimes \can, f_{ij})$, with its canonical involution, is the direct product $(Q_i,\can)\times(Q_j,\can)$. 
The same argument as in the proof of~\cite[Prop.3.5]{dejaiffe:orthsums} shows that the embedding of the direct product 
$\left(Q_1\otimes Q_2,\can\otimes\can, f_{12} \right)\times\left(Q_3\otimes Q_4,\can\otimes\can, f_{34} \right)$ in $(A,\sigma,f)$ induces an embedding of 
the tensor product of the corresponding Clifford algebras 
$$  \left( (Q_1,\can)\times (Q_2,\can)\right)\otimes \left((Q_3,\can)\times (Q_4,\can)\right)\hookrightarrow \left( \C(A,\sigma,f), \underline{\sigma}\right)\,.  $$
It follows that $c(u_1\otimes u_2)$ is $c_{12}(u_1\otimes u_2)\otimes (1,1)$, where $c_{12}$ is the canonical map from $(Q_1\otimes Q_2,\can\otimes \can, f_{12})$ to its Clifford algebra $(Q_1,\can)\times (Q_2,\can)$. This map is described in~\cite[(8.19)]{KMRT}, and we get 
\[c(u_1\otimes u_2)=(u_1,u_2)\otimes(1,1)=(u_1\otimes 1,u_2\otimes 1,u_1\otimes 1,u_2\otimes 1),\]
in $(Q_1\otimes Q_3)\times (Q_2\otimes Q_4)\times (Q_1\otimes Q_4)\times (Q_2\otimes Q_3).$
Therefore, the canonical semi-trace acts on $\Sym(Q_i\otimes Q_k,\can\otimes \can)$ for $i\in\{1,2\}$ and $k\in\{3,4\}$ by 
\[\underline f(x)=\Trd_{Q_i\otimes Q_k}((u_i\otimes 1)x).\] In particular, it vanishes on $\Sym(Q_i,\can)\otimes \Sym(Q_k,\can)$ and coincides with $f_{ik}$ by~\cite[(5.20)]{KMRT}, see also \cref{prod.section}. This finishes the proof.
\end{proof}

\begin{remark}
\label{isotropic-triples}
(1) If one of the quaternion algebras, say $Q_4$, is split, then the algebras with quadratic pair $(Q_i\otimes Q_4, \can\otimes\can, f_{i4})$ are hyperbolic for $i\in\{1,2,3\}$ and we get an isotropic trialitarian triple in which the algebras with quadratic pairs are respectively Witt equivalent to 
\[(Q_1\otimes Q_2,\can\otimes\can, f_{12}),\ \ (Q_1\otimes Q_3,\can\otimes\can, f_{13})\mbox{ and \ }(Q_2\otimes Q_3,\can\otimes \can, f_{23}),\]
where $Q_1$ $Q_2$ and $Q_3$ are quaternion algebras with $Q_1\otimes Q_2\otimes Q_3$ split. 

(2) The argument in the proof of~\cite[Prop.6.12]{jinv} extends to this setting and it follows that all trialitarian triples which are isotropic are either as in (1) or with a split component. In the second case, they coincide up to permutation with 
\[\bigl(\Ad_{q\perp\HH},M_2(D),M_2(D)\bigr),\] where $q$ is an Albert form, $D$ is the corresponding biquaternion algebra, which may have index $1,2$ or $4$ depending on $q$, and $M_2(D)$ is endowed with its hyperbolic quadratic pair, see \cref{isohypo}.  
\end{remark} 

The purpose of the remaining part of this section is to prove the following theorem, which provides a description of all trialitarian triples including at least two algebras of Schur index at most $2$. 

\begin{thm}
\label{lowindextriples}
Let $(A,B,C)$ be a trialitarian triple over $F$ such that at least two of the algebras $A$, $B$ and $C$ have Schur index at most $2$. Then there exist $F$-quaternion algebras $Q_1,Q_2,Q_3$ and $Q_4$, with $Q_1\otimes Q_2\otimes Q_3\otimes Q_4$ split, such that 
$$(A,\sigma_A,f_A)\in \left(Q_1\otimes Q_2,\can\otimes\can, f_{12} \right)\boxplus
\left(Q_3\otimes Q_4,\can\otimes\can, f_{34} \right)\,, $$
 $$(B,\sigma_B,f_B)\in  
\left(Q_1\otimes Q_3,\can\otimes\can, f_{13} \right)
\boxplus \left(Q_2\otimes Q_4,\can\otimes\can, f_{24} \right),
 $$
$$\mbox{ and }\ (C,\sigma_C,f_C)\in   
\left(Q_1\otimes Q_4,\can\otimes\can, f_{14} \right)
\boxplus \left(Q_2\otimes Q_3,\can\otimes\can, f_{23} \right)
\,. $$
\end{thm}
In order to prove the theorem, we use the so-called generalised quadratic forms, as defined by Tits in~\cite{Tits:genquadforms}. We first recall the definition and a few well-known facts. 
Let $A$ be a central simple $F$-algebra with involution $\theta$ of the first kind. 
A generalised quadratic form over $(A,\theta)$ is a pair $(V,q)$ where $V$ is a finite-dimensional right projective $A$-module  and $q$ is a map $q:V\rightarrow A/\Symd(A,\theta)$ subject to the following conditions:
\begin{enumerate}[$(a)$]
\item  $q(xd)=\theta(d)q(x)d$ for all $x\in V$ and $d\in A$.
\item There exists an hermitian form $h$ defined on $V$ and with values in $(A,\theta)$ such that for all $x,y\in V$ we have $q(x+y)- q(x)-q(y) = h(x,y) +\Symd(A,\theta)$. 
\end{enumerate}
In this case the hermitian form $(V,h)$ is uniquely determined (see \cite[(5.2)]{Knus:1991}) and we call it the {polar form of $(V,q)$}. Note that it follows from $(b)$ that \[h(x,x)\in\Symd(A,\theta),\mbox{ for all  }x\in V,\] hence the polar form of any quadratic form over $(A,\theta)$ is  alternating. 
We call $(V,q)$ {nonsingular} if  its polar form is nondegenerate. %and \emph{singular} otherwise.
%If the polar form of $(V,q)$ is identically zero we call $(V,q)$ \emph{totally singular}.
%We call $\dim_D(V)$ the \emph{dimension of  $(V,q)$} and denote it by $\dim_D(V,q)$.
We say that $(V,q)$   {represents} an element $a\in A$ if $q(x)=a+\Symd(A,\theta)$ for some $x\in V\setminus \{0\}$.
  We call $(V,q)$ {isotropic} if it represents $0$  and {anisotropic} otherwise.
For a field extension $K/F$ we write $(A,\theta)_K=(A\otimes_FK,\theta\otimes\id)$, $V_K=V\otimes_FK$,  and by $q_K$ we mean the unique  quadratic form $q:V_K\rightarrow A_K$  such that $q(x\otimes k)= q(x)k^2$ for all $x\in V$ and $k\in K$.

Let $D$ be a central simple  $F$-division algebra with involution of the first kind $\theta$. For $a_1,\ldots,a_n\in D$, we denote by $\qf{a_1,\ldots,a_n}$ the quadratic form $(D^n,q)$ over $(D,\theta)$ where  $q:D^n\rightarrow D$ is given by $$(x_1,\ldots, x_n)\mapsto \sum_{i=1}^n \theta(x_i)a_ix_i +\Symd(D,\theta)\, .$$ 
We call such a form a {diagonal form}. We call a quadratic form {diagonalisable} if it is isometric to a diagonal form. 

We will use the following: 

\begin{lem}\label{lem:genqfs} Assume $F$ is of characteristic $2$, and $Q$ is an $F$-quaternion division algebra endowed with its canonical involution. 
Let 
 $u\in Q$ be such that $u^2+u=a\in F$, and consider the quadratic extension $K=F[\wp^{-1}(a)]=F[\alpha]$, with $\alpha^2+\alpha=u^2+u=a$.  
If a generalised quadratic form $(V,q)$ over $(Q,\can)$ is isotropic over $K$,  then $q$ represents $u\lambda$ for some $\lambda\in F^\times$. 
  \end{lem}
  
  \begin{remark}
This result also holds in characteristic not 2, by excellence of quadratic field extensions and~\cite[Prop. p.382]{QT-daca}.
\end{remark} 
  
\begin{proof}
If $q$ is isotropic, then it splits off a $2$-dimensional isotropic form  isomorphic to  $\qf{u,-u}$, as all $2$-dimensional isotropic planes are isomorphic (see \cite[(5.6.1)]{Knus:1991}).
We may therefore assume that $q$ is anisotropic. 
Recall $\Symd(Q,\can)=F$ and $\Symd(Q_K,\can)=K$, so that $q$ and $q_K$ respectively have values in $Q/F$ and $Q_K/K$.

By assumption, the generalised quadratic space $(V_K,q_K)$ is isotropic, that is there exists $(x,y)\in V^2$ such that 
\[q_K(x\otimes 1+y\otimes \alpha)= 0\in Q_K/K.\] 
On the other hand, we have 
\begin{eqnarray*}
q_K(x\otimes 1 + y\otimes \alpha) & =& q_K(x\otimes 1) + q_K(y\otimes \alpha) + h_K(x\otimes 1,y\otimes \alpha) \\ 
&=& q(x)\otimes 1  + q(y)\otimes\alpha^2   + h(x,y) \otimes\alpha \\
&=& \bigl(q(x) + q(y) a\bigr)\otimes 1 + \bigl(h(x,y) + q(y))\bigr )\otimes\alpha\,.
\end{eqnarray*}
Since $Q_K/K\simeq Q/F\otimes 1\oplus Q/F\otimes \alpha$, we get that 
\[q(x)+q(y)a= 0\in Q/F\mbox{ and }h(x,y)+q(y)=0\in Q/F.\]
Consider now the element $z=x+yu\in V$. We have 
\begin{eqnarray*}
q(z)&=&q(x)+q(yu)+h(x,yu)=q(x)+(u+1)q(y)u+h(x,y)u\\
&=&\bigl(q(x)+uq(y)u\bigr)+\bigl(q(y)+h(x,y)\bigr)u\\
&=&\bigl(q(y)a+uq(y)u\bigr)+\bigl(q(y)+h(x,y)\bigr)u\in Q/F
\end{eqnarray*}
Take any $\xi\in Q$ such that $q(y)\equiv \xi\mod F$. Note that $\xi\not\equiv 0\mod F$ by the ansotropy of $q$.
The quaternion $\xi a+u\xi u$ commutes with $u$, hence belongs to $F[u]\subset Q$. In addition, we have proved that $q(y)+h(x,y)= 0\in Q/F$. 
It follows that $q(z)\in F[u]/F\subset Q/F$ and this proves the lemma. 
 \end{proof}
With this in hand, we may now prove~\cref{lowindextriples}. 
\begin{proof}
If the characteristic of $F$ is different from $2$, then the result follows from \cite[(6.2)]{QT:Arason12}  and the uniqueness of the semi-trace in this case. The result for fields of characteristic $2$ is similar, but we include the details below for completeness. 

Let $(A,\sigma,f)$ be an algebra with quadratic pair. We assume $A$ has degree $8$, $(\sigma,f)$ has trivial discriminant and two of $A$, $\C^+(A,\sigma,f)$ and $C^-(A,\sigma,f)$ have index at most $2$. By triality, we may assume that $A$ and at least one component of the Clifford algebra of $(A,\sigma,f)$ has index at most $2$. We claim there exists quaternion algebras $Q_1,Q_2,Q_3$ and $Q_4$ such that 
\begin{equation}
\label{claim}
 (A,\sigma,f) \in \left(Q_1\otimes Q_2,\can\otimes\can, f_{12} \right)\boxplus \left(Q_3\otimes Q_4,\can\otimes\can, f_{34} \right)\,.
 \end{equation}
The theorem follows from this claim by \cref{3orthsums}. 

We first consider the case where $A$ is split. Then $(A,\sigma,f)\simeq \Ad_q$ for some nonsingular  quadratic form $q$ over $F$ with $\Delta(q)=0$ and $\mathrm{ind}(\C_0(q))\leqslant 2$.  Therefore by \cref{prop:8dim}  there exist $a_1,_2,a_3,a_4\in F^\times$ and a $2$-dimensional nonsingular quadratic form $\phi$ over $F$ such that $q\simeq \qf{a_1,a_2,a_3,a_4}^{bi}\otimes \phi$. In particular we may write $q\simeq  \qf{a_1,a_2}^{bi}\otimes \phi\perp  \qf{a_3,a_4}^{bi}\otimes \phi$. Since these summands are similar to $2$-fold Pfister forms,  taking 
the adjoint quadratic pair now gives the result.

Now assume $\mathrm{ind}(A)=2$ and let
 $Q$ be an $F$-quaternion  division algebra such that $A$ is Brauer equivalent to $Q$.  
We may choose an $F$-basis $(1,u,v,w)$ of $Q$ such that $u^2+u=a\in F$, $v^2=b\in F^\times$ and $w=uv=v(u+1)$ for some  $a\in F$  and $b\in F^\times$. 
To prove \eqref{claim} in this case, we use generalised quadratic forms. Recall one constructs the adjoint quadratic pair to a generalised quadratic form $q$ over $(D,\theta)$  in an analogous way to the construction of an adjoint quadratic pair to a quadratic form over a field (see \cite[p.372]{tignol:qfskewfield}). We denote this quadratic pair by $\Ad_{q}$. 
Moreover, by \cite[(1.5)]{tignol:qfskewfield}, there exists a nonsingular generalised quadratic form $q$ over $(Q,\can)$ such that $\Ad_q\simeq (A,\sigma,f)$. 
 Since this form is nonsingular, it is diagonalisable by \cite[(6.3)]{dolphin:singuni}, so  there exist $u_1,u_2,u_3, u_4\in Q$ such that $q\simeq \qf{u_1,u_2,u_3, u_4}$.  Further, since $q$ is nonsingular,  we have that  $\Trd_Q(u_i)\neq 0$ for $i=1,\ldots, 4$ by \cite[(7.5)]{dolphin:singuni}.  Note also, that since $q$ is a map to $Q/F$, we may assume each  $u_i$  is in the $F$-subvector space of $Q$ generated by $u,v$ and $w$. In particular, since the $u_i$ have non zero trace, 
via changing the basis of $Q$ if necessary, we may assume that $u_1=u$.

Let $K=F(\wp^{-1}(a))$, a separable quadratic extension. We have that $$(A,\sigma,f)\in \Ad_{\qf{u}}\boxplus \Ad_{\qf{u_2,u_3,u_4}}\,.$$
After extending scalars to $K$, $\Ad_{\qf{u}}$ becomes hyperbolic. It follows that $\Ad_{\qf{u_2,u_3,u_4}}$ becomes adjoint to a $6$-dimensional quadratic form with trivial discriminant, that is, an Albert form.  The even part of the Clifford algebra of this form is of index at most $2$, and hence this Albert form is isotropic (see \cite[(16.5)]{KMRT}). So \cref{lem:genqfs} applies, and we may assume that $u_2=u\lambda$ for some $\lambda\in F$. We get that the generalised quadratic form $q$ is 
$$q= \qf{u,u\lambda}\perp \qf{u_3,u_4}\,.$$
By \cite[(3.4)]{dolphin:totdecompsymp}, we have that 
$$\Ad_{\qf{u,u\lambda}} = \Ad_{\pff{\lambda}}\otimes \Ad_{\qf{u}}\,. $$
Further, since  we have that $$\mathrm{disc}(A,\sigma,f)=\mathrm{disc}(\Ad_{\pff{\lambda}}\otimes \Ad_{\qf{u}})=0$$
 and that $\mathrm{disc}$ is additive across orthogonal sums (see \cite[(7.14)]{KMRT}), it follows that  $\mathrm{disc}(\Ad_{\qf{u_3,u_4}})$ is also trivial.  Hence $\Ad_{\qf{u_3,u_4}}$ is  decomposable. Since any  decomposable quadratic pair on a biquaternion algebra is of the form $(H_1\otimes H_2,\can\otimes \can,f_\otimes)$ for some quaternion algebras $H_1$ and $H_2$,  by \cite[(15.12)]{KMRT}, we get \eqref{claim}, and this concludes the proof. 
\end{proof}

\section{Appendix: Canonical semi-trace on the full Clifford algebra of a quadratic form}

In this appendix, we will show how one can construct a canonical semi-trace on the full Clifford algebra of a quadratic form. This semi-trace will be closely related to the semi-trace constructed in \cref{sec:canqp}. If the field is of characteristic different from $2$, then the full Clifford algebra has a unique semi-trace if and only if the canonical involution is orthogonal. Therefore, throughout this section, we assume that $F$ is a field of characteristic $2$.

We first give a construction of the full Clifford algebra of a nonsingular quadratic form. For this, we use the following presentation of quaternion algebras. When the characteristic of $F$ is $2$, a quaternion algebra may be defined as $F$-algebra generated by two elements $r,s$  subject to $r^2,s^2\in F$  and $rs+sr=1$ . If $s^2\neq 0$, then $(1,sr,s,sr^2)$ is a quaternion basis of this algebra, in the sense of \S 2.1, and otherwise the algebra is split (see \cite[p.25]{KMRT}). 

\begin{ex}\label{fullcliff}
Let  $q$ be a nonsingular quadratic form over $F$ with polar form $b$. Pick a decomposition \[q \simeq [a_1,b_1]\perp\ldots\perp [a_m,b_m],\] and let $(e_i,e_i')_{1\leqslant i \leqslant m}$ be the corresponding symplectic basis of the underlying vector space $V$. That is,  for all $i$ with $1\leqslant i \leqslant m$, we have  $q(e_i)=a_i$, $q(e_i')=b_i$, $b(e_i,e'_i)=1$ and $b(e_i,e_j)=b(e'_i,e'_j)=b(e_i,e'_j)=0$ for all $i\neq j$. We may assume $a_i\not=0$ for all $i$. 

 The full Clifford algebra of $q$ is generated by the elements $\{e_i,e'_i\}_{1\leqslant i\leqslant m}$, subject to the following relations for all $i\in\{1,\ldots, m\}$: 
\[\mbox\ e_i^2=a_i,\ {e_i'}^2=b_i,\ e_ie_i'+e_i'e_i=1.\] 
In addition, any pair of elements in the basis other than $(e_i,e_i')$ commute. By definition, the elements $e_i$ and $e_i'$ are fixed under the canonical involution $\underline{\sigma_q}$ on $\C(q)$. Therefore, the pairs $(e_i,e_i')_{1\leqslant i \leqslant m}$ each generate pairwise commuting $\underline{\sigma_q}$-stable $F$-quaternion subalgebras of $\C(q)$, respectively isomorphic to $[a_ib_i,a_i)$, and $\underline{\sigma_q}$ restricts to the canonical involution on each of these quaternion subalgebras. In particular, the canonical involution on $\C(q)$ is always  symplectic. 
\end{ex}
 
\begin{prop}
Let $(V,q)$ be a nonsingular quadratic space of even dimension $2m\geq 6$. 
Given a pair $(e,e')\in V^2$ with $b_q(e,e')=1$, the map 
\[f:\Sym(\C(q),\underline{\sigma_q})\rightarrow F,\ \ x\mapsto \Trd_{\C(q)}(ee'x),\]
is a semi-trace and does not depend on the choice of $(e,e')$. 
\end{prop} 
We will refer to this semi-trace as the canonical semi-trace on the full Clifford algebra, and  use the same notation, $\underline{f_q}$, as for the canonical semi-trace on the even Clifford algebra. 
\begin{remark}
\begin{enumerate}
\item If $m$ is even, we may define $\underline{f_q}$ as the semi-trace on the full Clifford algebra $\C(q)$ induced by the canonical semi-trace on $\C_0(q)$, in the sense of \S~\ref{prod.section}. 
Note though that restricting the canonical semi-trace of $\C(q)$, viewed as a map, to the even part $\Sym(\C_0(q),\underline{\sigma_q})$ does not produce a semi-trace, since the values are in $F$, while the centre of $\C_0(q)$ is a quadratic \'etale extension of $F$. 
\item Since the involution $\underline{\sigma_q}$ on $\C(q)$ always is symplectic, $m$ need not be even here, so this remark is not enough to prove the proposition. 
\end{enumerate}
\end{remark}
\begin{proof}

Let $(e,e')\in V^2$ be two vectors such that $b_q(e,e')=1$. One computes that  the element $u=ee'\in\C(q)$ satisfies $u+\underline{\sigma_q}(u)=1$, hence it determines a semi-trace $f_{e,e'}$ on $\C(q)$. 

If $m\geq 4$, the same computations as in \cref{split.rem} shows that $f_{e,e'}$ does not depend on the choice of $(e,e')$. If $m=3$ we may argue as follows. There exists a symplectic base $\varepsilon=(e,e',e_2,e_2',e_3,e_3').$ The computation at the end of \cref{split.rem} shows that $f_{e,e'}=f_{e_2,e_2'}=f_{e_3,e_3'}$. 
Consider now another pair of vectors $(g,g')$ with $b_q(g,g')=1$ which is part of a symplectic base $\delta=(g,g',g_2,g_2',g_3,g_3')$. 
By Revoy's Proposition~\cite[Prop. 3]{revoy:chainlemma}, there is a chain of symplectic bases $\eta_1=\varepsilon,\eta_2,\dots,\eta_r=\delta$ of $(V,q)$ such that for all $i$, $1\leq i\leq r-1$, the bases $\eta_i$ and $\eta_{i+1}$ have a common symplectic pair. Hence, again by the computations of \cref{split.rem}, the semi-traces $f_{\eta_i}$ and $f_{\eta_{i+1}}$ coincide, and this concludes the proof. 
\end{proof}

\begin{remark}\label{rem:fullcliffcan} Using \cref{fullcliff} and 
 the same arguments as in the proof of \cref{splitdec.prop}, we see that if $q=[a_1,b_1]\perp\dots\perp[a_m,b_m]$, with $m\geq 3$, 
then 
\[\bigr(\C(q),\underline{\sigma_q},\underline{f_q}\bigl)\simeq \bigl([a_1b_1,a_1)\otimes\dots\otimes[a_mb_m,a_m),\can\otimes\dots\otimes\can,f_\otimes\bigl).\]
In particular, two decompositions of the algebra with involution $(\C(q),\underline{\sigma_q})$ arising from two different presentations of $q$ give rise to the same canonical semi-trace. Compare with \cref{differentdecomp}. 
\end{remark}

\begin{cor}
Let $q$ be a nonsingular  quadratic form over $F$.
If $q$ is isotropic then  $(\C(q),\underline{\sigma_q},\underline{f_q})$ is hyperbolic.
\end{cor}
\begin{proof} 
This follows from \cref{rem:fullcliffcan} using a similar argument as to that in \cref{isohypo}.
\end{proof}

\end{document}